\documentclass[11pt]{amsart}

\usepackage[english]{babel}
\usepackage[T1]{fontenc}
\usepackage[utf8]{inputenc}
\usepackage{libertinus}
\usepackage[parfill]{parskip}
\usepackage[margin=1in]{geometry}

\usepackage{amsmath, amsthm, amsfonts, amssymb, mathtools}
\usepackage{bbm}        % Blackboard 1 font
\usepackage{enumerate}

\usepackage{graphicx}
\usepackage{tikz}
\usepackage{tikz-cd}
\usetikzlibrary{
  arrows.meta,
  graphs,
  decorations.pathmorphing,
  decorations.markings,
  hobby,
  fit,
  positioning
}
\usepackage{subcaption}
\usepackage{float}

\usepackage[
  colorlinks = true,
  linkcolor  = blue,
  citecolor  = blue,
  urlcolor   = blue
]{hyperref}
\usepackage[noabbrev,nameinlink,capitalise]{cleveref}
\usepackage{thmtools}

\usepackage{xargs}
\usepackage[dvipsnames]{xcolor}

\usepackage{enumitem}
\setlist{itemsep=0.2em, parsep=0.2em} % Adjust these values as needed
\setlist[itemize,1]{label=\ensuremath{\blacktriangleright}}
\setlist[itemize,2]{label=\ensuremath{\triangleright}}

\theoremstyle{plain}
\newtheorem{theorem}{Theorem}[section]
\newtheorem{lemma}[theorem]{Lemma}
\newtheorem{proposition}[theorem]{Proposition}
\newtheorem{corollary}[theorem]{Corollary}

\newtheorem*{thmA}{Theorem A}
\newtheorem*{thmB}{Theorem B}
\newtheorem*{thmC}{Theorem C}

\theoremstyle{definition}
\newtheorem{definition}[theorem]{Definition}

\newtheorem{example}[theorem]{Example}
\newtheorem{remark}[theorem]{Remark}

\newcommand{\Zp}{\mathbb{Z}_p}                        % $p$-adic integers
\newcommand{\Qp}{\mathbb{Q}_p}                        % $p$-adic numbers  
\newcommand{\valpi}{\mathrm{val}_{\unif}}                   % Valuation with respect to uniformizer
\newcommand{\Ring}{R}                          % Discrete valuation ring
\newcommand{\unif}{\pi}                        % Uniformizer

\newcommand{\sheaf}[1]{\mathcal{#1}}                  % Sheaf symbol
\newcommand{\stalk}[2]{#1(#2)}                        % Stalk notation: \stalk{\sheaf{F}}{v}
\newcommand{\face}{\to}                               % Incidence relation
\newcommand{\res}[3]{#1_{#2 \face #3}}                  % Restriction: \res{\sheaf{F}}{v}{e}

\newcommand{\sat}{\mathrm{sat}}
\newcommand{\free}{\mathrm{free}}
\newcommand{\tors}{\mathrm{tors}}

\newcommand{\Bock}{\beta}            % Bockstein on mod-π cohomology
\newcommand{\Conn}[1]{\Delta_{#1}}   % Connecting H^i(F/π^k) → H^{i+1}(F)
\newcommand{\Digit}[1]{\partial_{#1}}% Digit map H^i(F/π^k) → H^{i+1}(F/π)

\newcommand{\coboundary}{d}                           % Coboundary operator
\newcommand{\gr}{\mathrm{gr}}                         % Associated graded

\DeclareMathOperator{\im}{im}                         % Image
\DeclareMathOperator{\coker}{coker}                   % Cokernel  
\DeclareMathOperator{\Tor}{Tor}                       % Tor functor
\DeclareMathOperator{\rank}{rank}

\title[Precision-Graded Cohomology]{Precision-Graded Cohomology and Arithmetic Persistence \\ for Network Sheaves}
\author{Robert Ghrist}
\address{Departments of Mathematics and Electrical \& Systems Engineering \\ University of Pennsylvania \\ Philadelphia, PA 19104}
\email{ghrist@math.upenn.edu}
\author{Cassie Ding}
\address{Department of Electrical \& Systems Engineering \\ University of Pennsylvania \\ Philadelphia, PA 19104}
\email{dkexin@seas.upenn.edu}

\begin{document}

\begin{abstract}
Persistent homology tracks topological features across geometric scales, encoding birth and death of cycles as barcodes. We develop a complementary theory where the filtration parameter is algebraic precision rather than geometric scale. Working over the $p$-adic integers $\mathbb{Z}_p$, we define \emph{arithmetic barcodes} that measure torsion in network sheaf cohomology: each bar records the precision threshold at which a cohomology class fails to lift through the valuation filtration $\mathbb{Z}_p \supseteq p\mathbb{Z}_p \supseteq p^2\mathbb{Z}_p \supseteq \cdots$.

Our central result -- the \emph{Digit-SNF Dictionary} -- establishes that hierarchical precision data from connecting homomorphisms between successive mod-$p^k$ cohomology levels encodes exactly the Smith normal form exponents of the coboundary operator. Bars of length $a$ correspond to $\mathbb{Z}_p/p^a\mathbb{Z}_p$ torsion summands. For rank-one sheaves, cycle holonomy (the product of edge scalings around loops) determines bar lengths explicitly via $p$-adic valuation, and threshold stability guarantees barcode invariance when perturbations respect precision. Smith normal form provides integral idempotents projecting onto canonical cohomology representatives without geometric structure.

Results extend to arbitrary discrete valuation rings, with $p$-adic topology providing ultrametric geometry when available. Applications include distributed consensus protocols with quantized communication, sensor network synchronization, and systems where measurement precision creates natural hierarchical structure. The framework repositions torsion from computational obstacle to primary signal in settings where data stratifies by precision.
\end{abstract}


\maketitle

%--------------------------------------------------------------------
\section{Introduction}
\label{sec:intro}
%--------------------------------------------------------------------

Network sheaves are cellular sheaves \cite{Curry_SheavesCosheaves_2014} with a 1-dimensional base space.
These provide a flexible framework for encoding local data and 
constraints on graphs, with cohomology measuring global obstructions to 
consistency. Over fields -- typically $\mathbb{R}$, $\mathbb{C}$, or finite 
fields -- the theory is well-developed: cohomology groups are vector spaces, 
and when inner products are given, harmonic analysis via Hodge 
decomposition provides canonical representatives for cohomology classes. 

We consider the cohomology of network sheaves on finite graphs whose stalks are finite-rank 
free modules over a discrete valuation ring (DVR) -- a principal ideal 
domain $\Ring$ with unique nonzero prime ideal $(\unif)$, where $\unif$ is the 
\emph{uniformizer}, and with finite residue field $\Ring/\unif$. 
The prototypical example of a DVR to which our methods apply is the ring of $p$-adic 
integers $\Ring=\Zp$, where $\unif=p$ and $\Ring/\unif=\mathbb{F}_p$ \cite{AtiyahMacdonald_ICA_1969}. 
Other examples include formal power series rings $\mathbb{F}_q[[t]]$ and rings of integers 
in finite extensions of $\Qp$.

Over a DVR, cohomology changes character. The group $H^1$ generically 
carries \emph{torsion} -- elements annihilated by powers of $\unif$ -- even when 
its rationalization $H^1 \otimes_{\Ring} K$ (where $K = \mathrm{Frac}(\Ring)$) is 
free. This torsion is not pathological but rather the essential feature 
distinguishing integral from rational cohomology. Classical decompositions based 
on inner products fail because DVRs lack positive-definite pairings: orthogonal 
complements need not split, and $\im d$ need not be a direct summand of $C^1$. 
The obstacle is precisely the torsion in $H^1 = C^1/\im d$. 

The algebraic structure of DVRs provides natural tools for understanding this 
torsion. The filtration by powers of $\unif$ stratifies modules by precision: 
$M \supseteq \unif M \supseteq \unif^2 M \supseteq \cdots$, with successive 
quotients $\unif^k M/\unif^{k+1}M$ capturing data at the $k$-th level of 
refinement. For the $p$-adics $\Ring = \Zp$, this hierarchy has geometric meaning. The $p$-adic integers carry a natural ultrametric absolute value 
$|x|_p = p^{-\valpi(x)}$ satisfying the strong triangle inequality 
$|x+y|_p \leq \max\{|x|_p,|y|_p\}$ with equality when $|x|_p \neq |y|_p$ \cite{Schikhof1985,Gouva2020}. 
Multiplication by $p$ shifts elements one level deeper in this hierarchy, just 
as multiplying a decimal by $10$ shifts digits. The filtration 
$C^1 \supseteq pC^1 \supseteq p^2C^1 \supseteq \cdots$ thus encodes not merely 
an algebraic tower but a geometric refinement by precision, where each quotient 
$p^kC^1/p^{k+1}C^1$ represents data at a specific resolution \cite{Schikhof1985,Gouva2020,Koblitz1984}.

This precision stratification organizes cohomology hierarchically. Passing from 
cochains to cohomology, one obtains connecting homomorphisms -- the 
\emph{digit maps} -- between successive levels, whose images measure obstructions 
to lifting cohomology classes through the filtration. The dimensions of these 
images encode the complete torsion structure, and the relationship between this 
hierarchical data and the algebraic invariants from Smith normal form is our 
central result: {\em the Digit-SNF Dictionary}.

This all might seem a thin scenario -- only $H^0$ and $H^1$ are of interest for network sheaves. 
Yet, as $H^0$ classifies global sections and $H^1$ the obstructions thereunto, there is something 
nontrivial to be gained, especially for problems of network consensus and local-to-global 
integration. Our theory will develop with such applications in mind, and a prototypical 
example will follow at the conclusion of the theory. 

The paper has three main themes. First, we show that the filtration by powers of 
$\unif$ organizes cohomology into a hierarchy of precision levels, with 
connecting homomorphisms between levels whose dimensions encode the complete 
torsion structure. Second, we prove that this hierarchical information coincides 
exactly with the algebraic invariants from Smith normal form, establishing a 
dictionary that makes torsion both conceptually transparent and algorithmically 
accessible. Third, we show that the Smith normal form exponents define an 
\emph{arithmetic barcode} where torsion summands correspond to finite bars 
measuring precision thresholds, and prove stability results ensuring these 
barcodes are robust under high-precision perturbations.

Section~\ref{sec:prelims} establishes foundations: discrete valuation rings, 
graphs, and network sheaves with free stalks. Section~\ref{sec:bockstein-digit} 
develops the precision hierarchy through digit sequences and proves spectral 
sequence collapse for graphs. Section~\ref{sec:snf} constructs explicit integral 
decompositions via Smith normal form and proves the Digit-SNF Dictionary. 
Section~\ref{sec:barcodes} develops the arithmetic barcode interpretation, 
determines stability, and 
provides explicit formulas via cycle holonomy. Section~\ref{sec:application} 
demonstrates applications to distributed consensus with quantized communication.

%------------------------------------------
\subsection{Related Work and Context}
\label{ssec:related-work}
%------------------------------------------

Network sheaves on graphs and their cohomology and Hodge theory over fields are now well established.
Applications of network sheaf cohomology with field coefficients include network coding \cite{GhristHiraoka2011_NOLTA_NetworkCodingSheaves}, 
distributed sensing \cite{HansenGhrist2019_ICASSP_LearnSheafLaplacians}, distributed optimization \cite{HansenGhrist2019_Allerton_SheafHomologicalOpt}, spectral graph theory \cite{Hansen_2019}, 
opinion dynamics \cite{hansen2021opinion}, sheaf neural networks \cite{hansen2020sheaf,BodnarEtAl2022_NeuralSheafDiffusion}, graphic statics \cite{cooperband2023cosheaf,cooperband2023homological}, protein folding \cite{HayesEtAl2025_PersistentSheafLaplacianProtein} origami \cite{CooperbandGhrist_2025_OrigamiCosheaf}, 
and visual paradoxes \cite{GhristCooperband_2025_ObstructionsToReality}: in all cases, the network sheaf encodes how local measurements relate 
across edges, and $H^1$ gives obstructions to the global sections captured by $H^0$. While there
are examples of network sheaves with non-field coefficients, these have often been either
non-algebraic (e.g., set-valued network sheaves for applications to logic \cite{Goguen1992Sheaves}, sensing \cite{Robinson2017_SheavesSensorIntegration} or Reeb graphs 
 \cite{deSilva2016Categorified}) or rather specialized and esoteric (such as order lattices \cite{GhristRiess2022_TarskiLaplacian,riess2022diffusion} or categories enriched over quantales
 \cite{GhristLopezNorthRiess_2025_CategoricalDiffusion}). The investigation of p-adic coefficients for network sheaves is both general and novel. 

In persistent homology, the standard barcode classification relies on working over a field; the one-parameter classification as a multiset of intervals is a manifestation of quiver-representation theory for the $A_n$ quiver over a field \cite{CarlssonDeSilva2010,Oudot2015,CrawleyBoevey2015}. Over non-field coefficients such as $\mathbb{Z}$ or more general PIDs, torsion appears and the structure is richer: the modules along a filtration carry $p$-power torsion and the output depends on coefficient choice \cite{EdelsbrunnerLetscherZomorodian2002,ZomorodianCarlsson2005,RomeroSinchiTorres2016,ObayashiYoshiwaki2023}. Spectral-sequence tools such as the Bockstein homomorphism have been used to lift mod-$p$ information and compute torsion \cite{RenWuWu2017}. Our setting -- network sheaves with stalks free over a DVR -- isolates this torsion phenomenon in dimension $1$ and packages it into \emph{arithmetic barcodes} whose bar lengths are the uniformizer-exponents.

Our work synthesizes classical algebraic techniques with sheaf-theoretic computation in a new setting. Smith normal form over principal ideal domains is a cornerstone of module theory \cite{Newman_IntegralMatrices_1972,Kannan1979,Rotman2009}, and its use in computational homology is classical \cite{ZomorodianCarlsson2005,EdelsbrunnerHarer_CT_2010,MischaikowMrozek2002_ComputationalHomology}.
The application to network sheaf cohomology -- yielding explicit integral idempotents that commute with reduction and provide canonical representatives -- appears not to have been remarked upon. The Bockstein spectral sequence arising from coefficient ring filtrations is standard in algebraic topology \cite{McCleary2000,Rotman2009}, and its collapse for graphs follows from dimension considerations. The explicit quantitative relationship we establish between digit map dimensions and Smith normal form exponents -- the {\em Digit-SNF Dictionary} -- packages this correspondence as a practical computational tool for the first time. The stability results we prove for arithmetic barcodes under perturbations of the coboundary are reminiscent of bottleneck distance theorems in persistent homology \cite{CohenSteiner2006,EdelsbrunnerHarer_CT_2010}, though the ultrametric setting requires distinct techniques.

The ultrametric structure of $\Zp$ suggests connections to hierarchical data analysis. Ultrametric spaces appear naturally in phylogenetics, taxonomy, clustering, and more \cite{KuchakiRafsanjani2012}. Our precision filtration $\{p^k\}$ is analogous to scale filtrations in persistent homology, raising the possibility of ``$p$-adic persistence'' where torsion at different precisions plays the role of features at different scales. This connection remains to be developed. In the literature on persistent homology, dependence on coefficient choices is tied precisely to torsion in filtered pairs \cite{ObayashiYoshiwaki2023}; the $p$-adic viewpoint here makes that dependence algorithmic and quantitative via valuation exponents.

In the special case of the constant sheaf with $\mathbb{Z}$-coefficients, the torsion part of the first cohomology is intimately related to the well-studied critical group (or Jacobian, or sandpile group) of the graph \cite{Biggs1999_CriticalGroup, Lorenzini1991_GroupLaplacian}. This finite abelian group, whose structure is determined by the Smith normal form of the graph Laplacian, is the central object in the Riemann-Roch theory for graphs \cite{BakerNorine2007_RiemannRochGraphs}. Our framework extends this notion to arbitrary network sheaves over DVRs, where the arithmetic barcode provides a refined, precision-graded analysis of a `twisted' critical group whose structure is determined by both the graph topology and the sheaf's restriction maps.

%------------------------------------------
\subsection{Main Results}
\label{ssec:main}
%------------------------------------------

We work over a discrete valuation ring $\Ring$ with uniformizer $\unif$, and write $\sheaf{F}$ for a network sheaf on a finite graph $G$ with stalks that are finite-rank free $\Ring$-modules. 
Our three main results are as follows.

\begin{thmA}[Digit-SNF Dictionary]
\label{thm:A-digit-snf}
Let $\coboundary: C^0(G;\sheaf{F})\to C^1(G;\sheaf{F})$ be the coboundary. There exist unimodular matrices $U, V$ (with entries in $\Ring$) such that
\[
U\,[\coboundary]\,V = \mathrm{diag}\bigl(\unif^{a_1},\ldots,\unif^{a_r},0,\ldots,0\bigr)
\]
is the Smith normal form of $\coboundary$, with $0\le a_1\le\cdots\le a_r$ and $r=\rank_{\Ring}(\coboundary)$. For each $k\ge 0$, the digit connecting homomorphism
\[
\Digit{k}: H^0\bigl(G;\sheaf{F}/\unif^k\sheaf{F}\bigr)\longrightarrow H^1\bigl(G;\sheaf{F}/\unif\sheaf{F}\bigr)
\]
arising from
$0\to \unif^k\sheaf{F}/\unif^{k+1}\sheaf{F}\to \sheaf{F}/\unif^{k+1}\sheaf{F}\to \sheaf{F}/\unif^k\sheaf{F}\to 0$
satisfies
\[
\dim_{\,\Ring/\unif}\im\bigl(\Digit{k}\bigr) = \#\{\,j: 1 \le a_j\le k\,\}.
\]
Consequently, the sequence $\bigl(\dim_{\mathbb{F}} \im(\Digit{k})\bigr)_{k\ge 0}$ is nondecreasing, stabilizes at $\#\{j: a_j \ge 1\}$ (the number of torsion summands), and the multiset $\{a_j : a_j \ge 1\}$ of torsion exponents is determined by
\[
\#\{j:\ a_j = \ell\} \;=\; \dim_{\mathbb{F}} \im(\Digit{\ell}) - \dim_{\mathbb{F}} \im(\Digit{\ell-1})
\quad\text{for }\ell\ge 1.
\]
Moreover, when the residue field $\Ring/\unif$ is finite
\[
\bigl|H^1(G;\sheaf{F})_{\tors}\bigr| = \bigl(\#(\Ring/\unif)\bigr)^{\sum_{j=1}^r a_j}.
\]
\end{thmA}

\begin{remark}
\label{rem:matrix-notation}
The cochain groups $C^i(G;\sheaf{F})$ are finite products of stalks, hence finite-rank 
free $\Ring$-modules. Via the natural coordinate projection bases, we have canonical 
identifications $C^0(G;\sheaf{F}) \cong \Ring^{n_0}$ and $C^1(G;\sheaf{F}) \cong \Ring^{n_1}$ 
where $n_i$ is the number of $i$-cells. Under these identifications, the coboundary operator 
$\coboundary: C^0 \to C^1$ is represented as an $n_1 \times n_0$ matrix, and 
we may (as above) write $[\coboundary]$ to emphasize the explicit representation. 
\end{remark}

\begin{thmB}[Saturation Splitting]
\label{thm:B-saturation}
Let $K=\mathrm{Frac}(\Ring)$. The saturation
\[
\sat\bigl(\im \coboundary\bigr) := \bigl(\im \coboundary \otimes_{\Ring} K\bigr)\cap C^1(G;\sheaf{F})
\]
is the unique minimal direct summand of $C^1(G;\sheaf{F})$ containing $\im \coboundary$. If $U,V$ are as in Theorem~A, they induce \textbf{integral idempotent matrices}
\[
\Pi_{\ker}\in \mathrm{End}\bigl(C^0(G;\sheaf{F})\bigr),\qquad
\Pi_{\sat},\Pi_{\free}\in \mathrm{End}\bigl(C^1(G;\sheaf{F})\bigr),\qquad
\Pi_{\free}:=\mathrm{Id}-\Pi_{\sat},
\]
such that
\begin{enumerate}[label=\textup{(\roman*)}]
\item $\im(\Pi_{\ker})=\ker \coboundary = H^0(G;\sheaf{F})$;
\item $\im(\Pi_{\sat})=\sat(\im \coboundary)$;
\item the quotient $q:C^1\to H^1$ restricts to an isomorphism $q|_{\im(\Pi_{\free})}:\im(\Pi_{\free})\xrightarrow{\;\cong\;} H^1(G;\sheaf{F})_{\free}$;
\item $\sat(\im \coboundary)/\im \coboundary \cong H^1(G;\sheaf{F})_{\tors}$.
\end{enumerate}
Moreover, for all $k\ge 1$, reduction modulo $\unif^k$ commutes with $\Pi_{\ker},\Pi_{\sat},\Pi_{\free}$.
\end{thmB}

\medskip

\begin{thmC}[Truncated Stability of Arithmetic Barcodes]
\label{thm:C-truncated}
Let $\coboundary, \coboundary': C^0(G;\sheaf{F}) \to C^1(G;\sheaf{F})$ be coboundary operators over a discrete valuation ring $(\Ring, (\unif))$ with Smith normal form exponent multisets $\{a_j(d)\}$ and $\{a_j(d')\}$ respectively. If 
\[
\coboundary \equiv \coboundary' \pmod{\unif^m}
\]
for some $m \ge 1$, then for every $k < m$ the digit connecting maps coincide:
\[
\Digit{k}(\coboundary) = \Digit{k}(\coboundary')
\quad\text{as maps}\quad
H^0(G;\sheaf{F}/\unif^k\sheaf{F}) \longrightarrow H^1(G;\sheaf{F}/\unif\sheaf{F}),
\]
hence $d_k(\coboundary) = d_k(\coboundary')$ for all $k < m$. Consequently, the multiplicities of all bar lengths $\ell < m$ coincide:
\[
\#\{j : a_j(\coboundary) = \ell\} = \#\{j : a_j(\coboundary') = \ell\}
\quad\text{for all}\quad 1 \le \ell < m.
\]
Equivalently, the truncated barcodes
\[
\mathrm{Bar}^1_\unif(\coboundary) \cap [0,m) = \mathrm{Bar}^1_\unif(\coboundary') \cap [0,m)
\]
are identical as multisets of intervals, and the truncated valuation persistence modules $\{V^1_k(\coboundary)\}_{0 \le k < m}$ and $\{V^1_k(\coboundary')\}_{0 \le k < m}$ are isomorphic.
\end{thmC}

These three results form a coherent story. Theorem A (Digit-SNF Dictionary) shows that hierarchical information across precision levels exactly encodes the algebraic invariants: the valuation filtration and Smith normal form are two views of the same data. Theorem B (Saturation Splitting) leverages this to construct canonical integral representatives for cohomology classes without any geometric structure. Theorem C (Truncated Stability) establishes that arithmetic barcodes are robust under perturbations: when two coboundaries agree to sufficient precision, their truncated barcodes coincide, making the invariants numerically stable and suitable for applications involving measured data.

All our results are constructive and algorithmic. Smith normal form over $\Ring$ can be computed via elementary row/column operations or Hensel lifting from mod $\unif$ \cite{Kannan1979}. The digit maps $\Digit{k}$ are computed as connecting homomorphisms in long exact sequences, requiring cohomology computations over the quotient rings $\Ring/\unif^k$ with dimensions measured over the residue field $\Ring/\unif$. The Digit-SNF Dictionary provides two routes to the same invariants: compute the sequence $\{\dim \im(\Digit{k})\}$ from cohomology at successive precisions, or compute the Smith form directly and read off exponents. This redundancy enables error-checking and offers flexibility based on the computational context.

%--------------------------------------------------------------------
\section{Preliminaries: Network Sheaves over a Discrete Valuation Ring}
\label{sec:prelims}
%--------------------------------------------------------------------

Experts can skip this recollection of basic facts. 

%------------------------------------------
\subsection{Discrete Valuation Rings}
\label{ssec:dvr}
%------------------------------------------

A \emph{discrete valuation ring} (DVR) is an integral domain $\Ring$ with exactly one nonzero prime ideal $(\unif)$, where $\unif$ is called the \emph{uniformizer}. Equivalently, a DVR is a Noetherian local domain of Krull dimension $1$ with principal maximal ideal; in particular, a local principal ideal domain that is not a field.
Every nonzero element $x \in \Ring$ can be written uniquely as $x = \unif^n u$ where $n \ge 0$ and $u \in \Ring^\times$ is a unit. The integer $n = \valpi(x)$ is the \emph{valuation} of $x$, extended by $\valpi(0) = \infty$. The valuation satisfies $\valpi(xy) = \valpi(x) + \valpi(y)$ and $\valpi(x+y) \ge \min\{\valpi(x), \valpi(y)\}$ with equality when $\valpi(x) \neq \valpi(y)$.
The quotient $\mathbb{F} := \Ring/\unif$ is a field, called the \emph{residue field}.\footnote{In our main examples ($\Zp$, $\mathcal{O}_K$, $\mathbb{F}_q[[t]]$) it is finite, but finiteness or positive characteristic is not required for the algebra developed here.}
The fraction field $K := \mathrm{Frac}(\Ring)$ extends the valuation by $\valpi(x/y) = \valpi(x) - \valpi(y)$, making $K$ a discrete valuation field. 
The valuation induces a non-Archimedean absolute value $|x| := c^{-\valpi(x)}$ for any fixed real base $c>1$, defining a metric on $K$ that satisfies the strong triangle inequality $|x+y| \le \max\{|x|,|y|\}$.

\begin{example}
\label{ex:dvr-examples}
Standard examples of DVRs include:
\begin{enumerate}[label=\textup{(\roman*)}]
\item The $p$-adic integers $\Ring = \Zp$ for a prime $p$, with uniformizer $\unif = p$ and residue field $\mathbb{F}_p$. Elements are infinite base-$p$ expansions $\sum_{k=0}^\infty a_k p^k$ with $a_k \in \{0,\ldots,p-1\}$. The completion of $\mathbb{Z}$ at $(p)$ yields the field $\Qp$ of $p$-adic numbers.

\item Formal power series $\Ring = \mathbb{F}_q[[t]]$ over a finite field, with uniformizer $\unif = t$ and residue field $\mathbb{F}_q$. The fraction field $K = \mathbb{F}_q((t))$ consists of Laurent series.

\item The ring of integers $\mathcal{O}_K$ in any finite extension $K/\Qp$ is a DVR. In the unramified or totally ramified cases, one can describe the uniformizer explicitly.
\end{enumerate}
\end{example}

\begin{remark}
Throughout this paper, we write $\Ring$ for a general DVR with uniformizer $\unif$ and residue field $\mathbb{F} = \Ring/\unif$. Results in §§\ref{sec:prelims}--\ref{sec:snf} hold for any DVR. For concreteness and to fix intuition, we often specialize notation to $\Ring = \Zp$, $\unif = p$, and $\mathbb{F} = \mathbb{F}_p$, though the proofs remain valid in full generality. 
\end{remark}

Finitely generated modules over a DVR have simple structure: any such module decomposes as $M \cong \Ring^b \oplus \bigoplus_{j=1}^r \Ring/\unif^{a_j}$, where $b = \rank_\Ring(M)$ and the exponents $0 < a_1 \le \cdots \le a_r$ are uniquely determined. The free part $\Ring^b$ and torsion part $\bigoplus_j \Ring/\unif^{a_j}$ are canonical direct summands, and the multiset $\{a_j\}$ records the invariant factors. For network sheaf cohomology, these invariant factors will become the arithmetic barcode.

%------------------------------------------
\subsection{Network Sheaves: Algebraic Foundations}
\label{ssec:sheaves}
%------------------------------------------

We recall the basic theory of network sheaves, following the notation of prior work. Throughout this section, $\Ring$ denotes a discrete valuation ring with uniformizer $\unif$.

\begin{definition}
\label{def:network-sheaf}
A \emph{network sheaf} $\sheaf{F}$ of $\Ring$-modules on a finite graph $G = (V,E)$ consists of:
\begin{itemize}
\item For each vertex $v \in V$, an $\Ring$-module $\stalk{\sheaf{F}}{v}$ called the \emph{vertex stalk}.
\item For each edge $e \in E$, an $\Ring$-module $\stalk{\sheaf{F}}{e}$ called the \emph{edge stalk}.
\item For each incidence $v \to e$, an $\Ring$-linear map $\res{\sheaf{F}}{v}{e}: \stalk{\sheaf{F}}{v} \to \stalk{\sheaf{F}}{e}$ called the \emph{restriction map}.
\end{itemize}
\end{definition}

% \begin{example}
% \label{ex:constant-sheaf}
% The \emph{constant sheaf} $\underline{\Ring}$ has $\stalk{\underline{\Ring}}{v} = \stalk{\underline{\Ring}}{e} = \Ring$ for all cells, with all restriction maps equal to the identity.
% \end{example}

\begin{definition}
\label{def:sheaf-morphism}
A \emph{morphism} $\phi: \sheaf{F} \to \sheaf{G}$ of network sheaves consists of $\Ring$-linear maps $\phi_\sigma: \stalk{\sheaf{F}}{\sigma} \to \stalk{\sheaf{G}}{\sigma}$ for each cell $\sigma$ such that restriction maps commute:
\[
\begin{tikzcd}
\stalk{\sheaf{F}}{v} \arrow{r}{\phi_v} \arrow{d}{\res{\sheaf{F}}{v}{e}} & \stalk{\sheaf{G}}{v} \arrow{d}{\res{\sheaf{G}}{v}{e}} \\
\stalk{\sheaf{F}}{e} \arrow{r}{\phi_e} & \stalk{\sheaf{G}}{e}
\end{tikzcd}
\]
\end{definition}

The cochains and cohomology of a network sheaf generalize the cellular theory:

\begin{definition}
\label{def:sheaf-cochains}
For a network sheaf $\sheaf{F}$ on $G$, the \emph{cochain groups} are:
\[
C^i(G; \sheaf{F}) = \prod_{\sigma \in \Sigma_i(G)} \stalk{\sheaf{F}}{\sigma}
\]
where $\Sigma_0(G) = V$ and $\Sigma_1(G) = E$. The \emph{coboundary} $\coboundary: C^0(G; \sheaf{F}) \to C^1(G; \sheaf{F})$ is defined explicitly on a $0$-cochain $s$ for each oriented edge $e:u\to v$ via:
\[
(\coboundary s)(e) = \res{\sheaf{F}}{u}{e}(s(u)) - \res{\sheaf{F}}{v}{e}(s(v)) .
\]
\end{definition}

Since $G$ is 1-dimensional, we have $C^2(G; \sheaf{F}) = 0$ (there are no 2-cells), hence $\coboundary^2 = 0$ trivially. The \emph{sheaf cohomology} $H^i(G; \sheaf{F})$ is defined as usual: $H^0 = \ker \coboundary$ and $H^1 = \coker\coboundary = C^1/\im \coboundary$.

\begin{remark}
\label{rem:H0-free-H1-tors}
Throughout this paper, we assume all stalks are finite-rank free $\Ring$-modules. Since $H^0(G;\sheaf{F}) = \ker \coboundary$ is a submodule of the free module $C^0(G;\sheaf{F})$ over the principal ideal domain $\Ring$, it is itself free and hence torsion-free. By contrast, $H^1(G;\sheaf{F}) = C^1(G;\sheaf{F})/\im \coboundary$ is a quotient of a free module and generically carries $\unif$-power torsion. The structure theorem for finitely generated modules over a PID gives $H^1 \cong \Ring^b \oplus \bigoplus_{j=1}^r \Ring/\unif^{a_j}$, where $b$ is the rank of the free part and the exponents $\{a_j\}$ record the torsion. This torsion is measured precisely by the Bockstein homomorphism and the Smith normal form exponents, which are the subjects of Sections~\ref{sec:bockstein-digit} and~\ref{sec:snf}.
\end{remark}

%--------------------------------------------------------------------
\section{Valuation Filtration and Digit Sequences}
\label{sec:bockstein-digit}
%--------------------------------------------------------------------

The multiplicative structure $\unif^k \Ring \subseteq \Ring$ induces a decreasing 
filtration on cohomology whose successive quotients encode the complete torsion 
structure. We develop the resulting long exact sequences and prove that for graphs, 
the associated Bockstein spectral sequence collapses at $E_2$, making all torsion 
information explicit at the first-page level \cite{McCleary2000}.

%------------------------------------------
\subsection{The Valuation Filtration}
\label{ssec:val-filtration}
%------------------------------------------

\begin{definition}
\label{def:val-filtration}
For $k \geq 0$, the \emph{valuation filtration} on cochains is
\[
C^i_{(\geq k)}(G; \sheaf{F}) := \unif^k C^i(G; \sheaf{F}),
\]
the submodule of cochains with all coordinates divisible by $\unif^k$. This yields 
a decreasing, separated, exhaustive filtration
\[
C^i = C^i_{(\geq 0)} \supseteq C^i_{(\geq 1)} \supseteq C^i_{(\geq 2)} \supseteq 
\cdots
\]
with successive quotients $C^i_{(\geq k)}/C^i_{(\geq k+1)} \cong C^i(G; \sheaf{F}/\unif\sheaf{F})$ 
canonically isomorphic as $\mathbb{F}$-vector spaces.
\end{definition}

\begin{lemma}
\label{lem:filtration-properties}
The valuation filtration satisfies:
\begin{enumerate}[label=\textup{(\roman*)}]
\item $\coboundary(C^i_{(\geq k)}) \subseteq C^{i+1}_{(\geq k)}$ for all $k$, 
hence each $C^\bullet_{(\geq k)}$ is a subcomplex.
\item The canonical isomorphism $\phi_k: \unif^k C^i/\unif^{k+1} C^i 
\xrightarrow{\cong} C^i/\unif C^i$ given by $\phi_k([\unif^k y]) = [y]$ is 
$\mathbb{F}$-linear and independent of $k$.
\item For $\Ring = \Zp$ with compatible norms, each $C^i_{(\geq k)}$ is a closed 
ball of radius $p^{-k}$, hence closed in the ultrametric topology.
\end{enumerate}
\end{lemma}

\begin{proof}
(i) Since $\coboundary$ is $\Ring$-linear and $\unif^k C^i = \{\unif^k s : s \in C^i\}$, 
we have $\coboundary(\unif^k s) = \unif^k \coboundary(s) \in \unif^k C^{i+1}$.

(ii) For any $[\unif^k y] \in \unif^k C^i/\unif^{k+1} C^i$, if $\unif^k y \equiv 
\unif^k y' \pmod{\unif^{k+1}}$, then $\unif^k(y-y') \in \unif^{k+1} C^i$. Since 
$C^i$ is free over $\Ring$, this implies $y - y' \in \unif C^i$, so $[y] = [y']$ 
in $C^i/\unif C^i$. The map is clearly surjective and $\mathbb{F}$-linear, with 
trivial kernel, hence an isomorphism. Independence from $k$ is manifest in the 
formula.

(iii) In the ultrametric setting, $C^i_{(\geq k)} = \{s \in C^i : \|s\| \leq p^{-k}\}$ 
is a closed ball, hence closed.
\end{proof}

%------------------------------------------
\subsection{Fundamental Short Exact Sequences}
\label{ssec:fundamental-ses}
%------------------------------------------

The multiplicative hierarchy of ideals $\Ring \supset (\unif) \supset (\unif^2) 
\supset \cdots$ yields a web of exact sequences of sheaves.

\begin{lemma}
\label{lem:sheaf-filtration-iso}
For any network sheaf $\sheaf{F}$ with stalks that are free $\Ring$-modules and for all $k \geq 0$, the map
\[
\phi_k: \unif^k\sheaf{F}/\unif^{k+1}\sheaf{F} \longrightarrow \sheaf{F}/\unif\sheaf{F}
\]
defined stalkwise by $\phi_k([\unif^k x]) = [x]$ is a natural isomorphism of network sheaves compatible with all restriction maps.
\end{lemma}

\begin{proof}
At each stalk $\stalk{\sheaf{F}}{\sigma}$, the map $\phi_{k,\sigma}: \unif^k\stalk{\sheaf{F}}{\sigma}/\unif^{k+1}\stalk{\sheaf{F}}{\sigma} \to \stalk{\sheaf{F}}{\sigma}/\unif\stalk{\sheaf{F}}{\sigma}$ given by $[\unif^k x] \mapsto [x]$ is well-defined: if $\unif^k x \equiv \unif^k x' \pmod{\unif^{k+1}}$, then $\unif^k(x - x') = \unif^{k+1} y$ for some $y$, hence $x - x' = \unif y$ since $\stalk{\sheaf{F}}{\sigma}$ is free over $\Ring$ and $\unif$ is not a zero-divisor. Thus $[x] = [x']$ in $\stalk{\sheaf{F}}{\sigma}/\unif\stalk{\sheaf{F}}{\sigma}$. The map is clearly $\mathbb{F}$-linear, surjective (since $[x]$ is the image of $[\unif^k x]$), and injective (since $[\unif^k x] \mapsto 0$ implies $x \in \unif\stalk{\sheaf{F}}{\sigma}$, hence $\unif^k x \in \unif^{k+1}\stalk{\sheaf{F}}{\sigma}$).

To verify compatibility with restrictions, consider an incidence $\tau \to \sigma$ and the restriction map $\res{\sheaf{F}}{\tau}{\sigma}: \stalk{\sheaf{F}}{\tau} \to \stalk{\sheaf{F}}{\sigma}$. Since restriction maps are $\Ring$-linear, they commute with multiplication by $\unif^k$ and pass to quotients. The diagram
\[
\begin{tikzcd}
\unif^k\stalk{\sheaf{F}}{\tau}/\unif^{k+1}\stalk{\sheaf{F}}{\tau} \arrow{r}{\phi_{k,\tau}} \arrow{d}{\res{\sheaf{F}}{\tau}{\sigma}} & \stalk{\sheaf{F}}{\tau}/\unif\stalk{\sheaf{F}}{\tau} \arrow{d}{\res{\sheaf{F}}{\tau}{\sigma}} \\
\unif^k\stalk{\sheaf{F}}{\sigma}/\unif^{k+1}\stalk{\sheaf{F}}{\sigma} \arrow{r}{\phi_{k,\sigma}} & \stalk{\sheaf{F}}{\sigma}/\unif\stalk{\sheaf{F}}{\sigma}
\end{tikzcd}
\]
commutes, since both paths send $[\unif^k x]$ to $[\res{\sheaf{F}}{\tau}{\sigma}(x)]$.
\end{proof}

\begin{proposition}
\label{prop:basic-ses}
For $k \geq 1$, the following are short exact sequences of network sheaves:
\begin{enumerate}[label=\textup{(\roman*)}]
\item $0 \to \sheaf{F} \xrightarrow{\unif^k} \sheaf{F} \to \sheaf{F}/\unif^k\sheaf{F} \to 0$
\item $0 \to \unif^k\sheaf{F}/\unif^{k+1}\sheaf{F} \to \sheaf{F}/\unif^{k+1}\sheaf{F} 
\to \sheaf{F}/\unif^k\sheaf{F} \to 0$
\end{enumerate}
where all maps are induced by the natural inclusions and quotient maps, applied 
stalkwise.
\end{proposition}

\begin{proof}
Exactness is checked stalkwise. Since each stalk is a free $\Ring$-module and 
$\unif$ is not a zero-divisor in $\Ring$, multiplication by $\unif^k$ is injective. 
The kernel of the quotient map $\sheaf{F} \to \sheaf{F}/\unif^k\sheaf{F}$ is 
precisely $\unif^k\sheaf{F}$ by definition. For sequence (ii), the third isomorphism 
theorem gives $(\sheaf{F}/\unif^{k+1}\sheaf{F})/(\unif^k\sheaf{F}/\unif^{k+1}\sheaf{F}) 
\cong \sheaf{F}/\unif^k\sheaf{F}$.
\end{proof}

%------------------------------------------
\subsection{Long Exact Sequences and the Bockstein Homomorphism}
\label{ssec:bockstein}
%------------------------------------------

Applying the cohomology functor to the short exact sequences of Proposition 
\ref{prop:basic-ses} yields connecting homomorphisms that organize the torsion 
structure.

\begin{theorem}[Long Exact Sequences]
\label{thm:les}
{\cite{Rotman2009,McCleary2000}}
Sequence (i) of Proposition \ref{prop:basic-ses} induces a long exact sequence
\[
0 \to H^0(G; \sheaf{F}) \xrightarrow{\unif^k} H^0(G; \sheaf{F}) \to 
H^0(G; \sheaf{F}/\unif^k\sheaf{F}) \xrightarrow{\Conn{k}} H^1(G; \sheaf{F}) 
\xrightarrow{\unif^k} H^1(G; \sheaf{F}) \to H^1(G; \sheaf{F}/\unif^k\sheaf{F}) \to 0
\]
where the final zero uses $H^2(G; \sheaf{F}) = 0$ since $G$ is a graph. The 
\emph{connecting homomorphism} $\Conn{k}$ takes values in the $\unif^k$-torsion 
submodule $H^1(G; \sheaf{F})[\unif^k] := \{x \in H^1 : \unif^k x = 0\}$.

Similarly, sequence (ii) induces
\[
\cdots \to H^0(G; \sheaf{F}/\unif^k\sheaf{F}) \xrightarrow{\Digit{k}} 
H^1(G; \unif^k\sheaf{F}/\unif^{k+1}\sheaf{F}) \to 
H^1(G; \sheaf{F}/\unif^{k+1}\sheaf{F}) \to H^1(G; \sheaf{F}/\unif^k\sheaf{F}) \to 0
\]
where $\Digit{k}$ is the \emph{digit connecting map}.
\end{theorem}

\begin{proof}
Standard functoriality of sheaf cohomology applied to the exact sequences of sheaves. 
The image of $\Conn{k}$ lies in $\ker(\unif^k: H^1 \to H^1) = H^1[\unif^k]$ by 
exactness.
\end{proof}

\begin{definition}
\label{def:bockstein-digit}
The \emph{Bockstein homomorphism} $\Bock: H^i(G; \sheaf{F}/\unif\sheaf{F}) \to 
H^{i+1}(G; \sheaf{F}/\unif\sheaf{F})$ is the composition
\[
H^i(G; \sheaf{F}/\unif\sheaf{F}) \xrightarrow{\Conn{1}} H^{i+1}(G; \sheaf{F})[\unif] 
\xrightarrow{\pi} H^{i+1}(G; \sheaf{F})/\unif H^{i+1}(G; \sheaf{F}) \cong 
H^{i+1}(G; \sheaf{F}/\unif\sheaf{F}),
\]
where $\pi$ is the quotient map. The isomorphism $H^{i+1}(G; \sheaf{F})/\unif H^{i+1}(G; \sheaf{F}) \cong H^{i+1}(G; \sheaf{F}/\unif\sheaf{F})$ follows from the long exact sequence associated with $0 \to \sheaf{F} \xrightarrow{\unif} \sheaf{F} \to \sheaf{F}/\unif\sheaf{F} \to 0$, using that $H^{i+2}(G;\sheaf{F}) = 0$ for graphs (so the connecting map from $H^{i+1}(G;\sheaf{F}/\unif)$ lands in zero).
\end{definition}

The key relationship between these connecting maps is:

\begin{proposition}[Digit-Bockstein Compatibility]
\label{prop:digit-bockstein}
Under the canonical isomorphism $\phi_k: \unif^k\sheaf{F}/\unif^{k+1}\sheaf{F} 
\xrightarrow{\cong} \sheaf{F}/\unif\sheaf{F}$ of Lemma~\ref{lem:sheaf-filtration-iso},
the digit map satisfies
\[
H^{i+1}(\phi_k) \circ \Digit{k} = \Bock \circ \rho_k
\]
where $\rho_k: H^i(\sheaf{F}/\unif^k) \to H^i(\sheaf{F}/\unif)$ is the natural 
reduction map. Consequently, for all $k \geq 0$,
\[
\im(\Digit{k}) \subseteq \im(\Bock) \subseteq H^{i+1}(G; \sheaf{F}/\unif\sheaf{F}),
\]
and the dimension bound follows from the chain
\[
\dim_{\mathbb{F}}\im(\Digit{k}) = \dim_{\mathbb{F}}\im(\Bock\circ \rho_k)
\le \dim_{\mathbb{F}}\im(\Bock)\le \dim_{\mathbb{F}} H^{i+1}(G;\sheaf{F})[\unif].
\]
\end{proposition}
\begin{proof}
The commutativity follows from naturality of connecting homomorphisms with respect 
to the commutative diagram of short exact sequences relating (i) and (ii) from 
Proposition \ref{prop:basic-ses} via $\phi_k$. The inclusion $\im(\Digit{k}) 
\subseteq \im(\Bock)$ is immediate, and $\im(\Bock) \subseteq H^{i+1}(\sheaf{F}/\unif)$ 
by definition. The dimension bound follows since all images lie in the same 
finite-dimensional $\mathbb{F}$-vector space $H^{i+1}(\sheaf{F}/\unif)$.
\end{proof}

\begin{remark}
\label{rem:digit-flag}
The digit maps form an ascending flag
\[
\im(\Digit{0}) \subseteq \im(\Digit{1}) \subseteq \im(\Digit{2}) \subseteq \cdots
\subseteq H^{i+1}(G; \sheaf{F}/\unif\sheaf{F})
\]
in the fixed $\mathbb{F}$-vector space $H^{i+1}(\sheaf{F}/\unif)$. Because the reduction maps $H^0(\sheaf{F}/\unif^{k+1}) \to H^0(\sheaf{F}/\unif^k)$ compose with the Bockstein as $\Bock \circ \rho_k = \Bock \circ \rho_{k+1} \circ (\text{reduction})$ (Proposition~\ref{prop:digit-bockstein}), we have $\im(\Digit{k}) \subseteq \im(\Digit{k+1})$. As $k$ increases, additional torsion summands ``switch on'' when $k$ reaches their bar length. This flag encodes the complete torsion structure, as Theorem~A will show.
\end{remark}

\begin{proposition}[Image of the Bockstein via UCT]
\label{prop:bockstein-image}
If $H^{i+1}(G;\sheaf{F}) \cong \Ring^{b_{i+1}} \oplus \bigoplus_{j=1}^{r_{i+1}} \Ring/\unif^{a_{i+1,j}}$, then
\[
\im\bigl(\Bock:H^i(G;\sheaf{F}/\unif)\to H^{i+1}(G;\sheaf{F}/\unif)\bigr)
\cong \Tor_1^\Ring\!\bigl(H^{i+1}(G;\sheaf{F}),\Ring/\unif\bigr)
\cong \mathbb{F}^{\,r_{i+1}}.
\]
In particular, $\dim_{\mathbb{F}}\im(\Bock)=r_{i+1}$.
\end{proposition}

\begin{remark}
\label{rem:digit-bockstein-relationship}
By Proposition \ref{prop:digit-bockstein}, each digit map $\Digit{k}$ factors through the 
Bockstein, hence $\im(\Digit{k}) \subseteq \im(\Bock)$ for all $k \ge 0$. The precise 
relationship between the dimensions $\dim_{\mathbb{F}}\im(\Digit{k})$ and the Smith 
normal form exponents $\{a_{i+1,j}\}$ is the content of Theorem~A 
(the Digit-SNF Dictionary). Lemmas~\ref{lem:two-term-digit} and~\ref{lem:digit-increment} 
establish that $d_k = \#\{j : 1 \le a_j \le k\}$, with increments $d_k - d_{k-1}$ counting 
bars of length exactly $k$.
\end{remark}

\begin{proof}
Because $C^\bullet(G;\sheaf{F})$ is a bounded complex of free $\Ring$-modules, \cite{Rotman2009}
the universal coefficient short exact sequence yields
\[
0\to H^i(G;\sheaf{F})\otimes_\Ring\mathbb{F}\to H^i(G;\sheaf{F}/\unif)
\xrightarrow{\;\Bock\;} \Tor_1^\Ring\!\bigl(H^{i+1}(G;\sheaf{F}),\mathbb{F}\bigr)\to 0.
\]
Now $\Tor_1^\Ring(\Ring/\unif^{a},\mathbb{F})\cong\mathbb{F}$ for every $a\ge 1$, and 
$\Tor_1^\Ring(\Ring,\mathbb{F})=0$, whence the claim.
\end{proof}

\begin{remark}
By Proposition \ref{prop:digit-bockstein}, each digit map $\Digit{k}$ factors through the 
Bockstein, hence $\im(\Digit{k}) \subseteq \im(\Bock)$ for all $k \ge 0$. The precise 
relationship between the dimensions $\dim_{\mathbb{F}}\im(\Digit{k})$ and the Smith 
normal form exponents $\{a_{i+1,j}\}$ is the content of Theorem~A 
(the Digit-SNF Dictionary). Lemmas~\ref{lem:two-term-digit} and~\ref{lem:digit-increment} 
establish that $d_k = \#\{j : 1 \le a_j \le k\}$, with increments $d_k - d_{k-1}$ counting 
bars of length exactly $k$.
\end{remark}

%------------------------------------------
\subsection{The Bockstein Spectral Sequence}
\label{ssec:bockstein-ss}
%------------------------------------------

The valuation filtration $F^k C^i := \unif^k C^i(G;\sheaf{F})$ is a decreasing, 
exhaustive, separated filtration by subcomplexes, hence defines a spectral sequence 
converging to the associated graded of cohomology.

\begin{theorem}[Bockstein Spectral Sequence]
\label{thm:bss}
{\cite{McCleary2000}}
The valuation filtration induces a first-quadrant spectral sequence with
\[
E_1^{i,k} \cong H^i(G; \sheaf{F}/\unif\sheaf{F}) \quad \text{for all } k \geq 0,
\]
whose $d_1$ differential is the Bockstein homomorphism
\[
d_1 = \Bock: H^i(G; \sheaf{F}/\unif\sheaf{F}) \longrightarrow 
H^{i+1}(G; \sheaf{F}/\unif\sheaf{F}).
\]
The spectral sequence converges to the associated graded
\[
E_\infty^{i,k} \cong \gr^k H^i(G; \sheaf{F}) = 
\frac{\unif^k H^i(G; \sheaf{F})}{\unif^{k+1} H^i(G; \sheaf{F})}.
\]
The differential $d_r: E_r^{i,k} \to E_r^{i+r,k-r+1}$ has bidegree $(r, -r+1)$.
\end{theorem}

\begin{proof}
Standard spectral sequence construction for a filtered complex. The graded piece 
$\gr^k C^i = C^i_{(\geq k)}/C^i_{(\geq k+1)} = \unif^k C^i/\unif^{k+1} C^i \cong 
C^i(\sheaf{F}/\unif)$ by Lemma~\ref{lem:sheaf-filtration-iso}, giving 
$E_1^{i,k} \cong H^i(\gr^k C^\bullet) \cong H^i(\sheaf{F}/\unif)$. The induced 
differential on $E_1$ is the connecting homomorphism from the short exact sequence
\[
0 \to \unif^{k+1}C^\bullet \to \unif^k C^\bullet \to C^\bullet(\sheaf{F}/\unif) \to 0,
\]
which is precisely the Bockstein $\Bock$. Convergence to $\gr^\bullet H^i$ follows 
from standard theory for bounded-below, exhaustive, separated filtrations over a PID \cite{McCleary2000,Rotman2009}.
\end{proof}

\begin{remark}[Collapse at $E_2$ for graphs]
\label{rem:spectral-sequence-pages}
For a graph, the cochain complex has nonzero terms only in degrees $i=0$ and $i=1$, 
hence the spectral sequence has only two rows: $E_r^{0,k}$ and $E_r^{1,k}$. The 
differential $d_r: E_r^{i,k} \to E_r^{i+r,k-r+1}$ increases the cohomological degree 
by $r$. For $r \ge 2$, the differential $d_r: E_r^{0,k} \to E_r^{r,k-r+1}$ and $d_r: E_r^{1,k} \to E_r^{1+r,k-r+1}$ both land in row $i \ge 2$, which is identically zero. 
Therefore all differentials $d_r$ for $r \ge 2$ are zero, and the 
spectral sequence collapses at $E_2$: $E_2 = E_3 = \cdots = E_\infty$. In particular, 
$E_\infty^{1,k}\cong \gr^k H^1(G;\sheaf{F})$.
\end{remark}

\begin{remark}
\label{rem:graded-structure}
The graded pieces $\gr^k H^1(G; \sheaf{F})$ encode the complete torsion structure. While the spectral sequence provides one perspective on this structure, Theorem~A establishes a direct relationship between digit map dimensions and Smith normal form exponents, bypassing spectral sequence computations entirely \cite{McCleary2000}.
\end{remark}

%------------------------------------------
\subsection{Completeness and Inverse Limits}
\label{ssec:completeness}
%------------------------------------------

We now confirm a foundational property that justifies our focus on the tower of 
mod-$\unif^k$ reductions: the cohomology over $\Ring$ is the inverse limit of the cohomology 
groups over the quotient rings, making it $\unif$-adically complete.
The inverse system $\{C^\bullet/\unif^k\}_{k\ge 1}$ of cochain complexes with 
bonding maps given by the natural surjections satisfies the Mittag-Leffler 
condition since the bonding maps are surjective. By Milnor's theorem on derived 
limits, $\varprojlim{}^{1} H^{i-1}(C^\bullet/\unif^k) = 0$ for all $i$, hence
\[
H^i\!\left(\varprojlim_k C^\bullet/\unif^k\right) \cong \varprojlim_k H^i(C^\bullet/\unif^k).
\]

\begin{proposition}
\label{prop:completeness}
If $\Ring$ is $\unif$-adically complete (e.g., $\Ring = \Zp$), then 
$C^\bullet(G;\sheaf{F}) \cong \varprojlim_k C^\bullet(G;\sheaf{F})/\unif^k$, 
and consequently
\[
H^i(G;\sheaf{F}) \cong \varprojlim_{k} H^i(G;\sheaf{F}/\unif^k\sheaf{F}).
\]
In particular, $H^i(G;\sheaf{F})$ is $\unif$-adically complete.
\end{proposition}

\begin{proof}
Since each $C^i(G;\sheaf{F})$ is a finite product of free $\Ring$-modules and 
$\Ring$ is complete, the natural map $C^i \to \varprojlim_k C^i/\unif^k C^i$ is 
an isomorphism. The inverse system $\{C^\bullet/\unif^k\}$ has surjective bonding 
maps, hence satisfies the Mittag-Leffler condition. By Milnor's theorem on derived 
limits \cite{Weibel1994}, $\varprojlim{}^{1} H^{i-1}(C^\bullet/\unif^k) = 0$ 
for all $i$, and the second statement follows.
\end{proof}

\begin{corollary}
By Nakayama's lemma for the local ring $(\Ring,(\unif))$, if $H^i(G;\sheaf{F})/\unif = 0$, then $H^i(G;\sheaf{F}) = 0$. More generally, a morphism $\phi:\sheaf{F}\to\sheaf{G}$ inducing a surjection on mod $\unif$ cohomology induces a surjection on full cohomology.
\end{corollary}

%------------------------------------------
\subsection{Lifting and Obstructions}
\label{ssec:lifting}
%------------------------------------------

The connecting homomorphisms measure obstructions to lifting cohomology classes 
through the valuation filtration.

\begin{theorem}[Lifting Criterion]
\label{thm:lifting}
For $\alpha \in H^i(G; \sheaf{F}/\unif\sheaf{F})$, the following are equivalent:
\begin{enumerate}[label=\textup{(\roman*)}]
\item $\alpha$ lifts to a class in $H^i(G; \sheaf{F})$.
\item $\alpha$ extends to a compatible system $\{\alpha_k \in H^i(G; \sheaf{F}/\unif^k\sheaf{F})\}_{k \geq 1}$ 
with $\alpha_1 = \alpha$ and each $\alpha_{k+1}$ mapping to $\alpha_k$ under reduction.
\item $\Bock(\alpha) = 0$ in $H^{i+1}(G; \sheaf{F}/\unif\sheaf{F})$.
\end{enumerate}
When these conditions hold, the lift to $H^i(G; \sheaf{F})$ is unique modulo 
$\unif H^i(G; \sheaf{F})$.
\end{theorem}

\begin{proof}
(i) $\Rightarrow$ (ii): If $\tilde{\alpha} \in H^i(G; \sheaf{F})$ lifts $\alpha$, 
set $\alpha_k := \tilde{\alpha} \bmod \unif^k$.

(ii) $\Rightarrow$ (iii): A compatible system shows that all connecting obstructions 
$\Digit{k}(\alpha_k) = 0$ vanish. By Proposition \ref{prop:digit-bockstein}, 
$\Bock(\alpha) = 0$.

(iii) $\Rightarrow$ (i): From the long exact sequence of Theorem \ref{thm:les}, 
$\ker(\Conn{1}) = \im(H^i(G; \sheaf{F}) \to H^i(G; \sheaf{F}/\unif\sheaf{F}))$. 
Since $\Bock = \pi \circ \Conn{1}$ and $\pi$ is surjective onto $H^{i+1}/\unif$, 
the condition $\Bock(\alpha) = 0$ implies $\Conn{1}(\alpha) \in \unif H^{i+1}$. 
But $\Conn{1}(\alpha) \in H^{i+1}[\unif]$ by Theorem \ref{thm:les}, hence 
$\Conn{1}(\alpha) = 0$, so $\alpha$ lifts.

Uniqueness follows since two lifts differ by an element of $\unif H^i(G; \sheaf{F})$.
\end{proof}

\begin{remark}
Since $H^i(G;\sheaf{F})$ is a finitely generated $\Ring$-module, its torsion 
submodule is finite. Thus only finitely many mod $\unif$ classes have nontrivial 
Bockstein, and for any class, there exists a maximal precision $k_0$ to which it 
lifts (or it lifts to all levels). This finite obstruction property is fundamental 
to the barcode interpretation of Section \ref{sec:barcodes}.
\end{remark}

\begin{remark}[Torsorial interpretation]
\label{rem:torsor-viewpoint}
For any abelian sheaf $\sheaf{F}$ on a graph $G$, isomorphism classes of 
$\sheaf{F}$-torsors are canonically classified by $H^1(G;\sheaf{F})$ (see, e.g., 
\cite{Stacks_Torsors_H1}). The short exact sequences
\[
0 \to \unif^k\sheaf{F}/\unif^{k+1}\sheaf{F} \to \sheaf{F}/\unif^{k+1}\sheaf{F} 
\to \sheaf{F}/\unif^k\sheaf{F} \to 0
\]
induce connecting maps that measure the obstruction to lifting 
$\sheaf{F}/\unif^k$-torsors to $\sheaf{F}/\unif^{k+1}$-torsors. The digit maps 
$\Digit{k}$ thus count, at each precision level, how many independent torsorial 
obstructions emerge. In this language, the nondecreasing sequence 
$\dim_{\mathbb{F}} \im(\Digit{k})$ tracks the accumulation of torsor-lifting 
obstructions as precision increases, and the arithmetic barcode becomes a 
torsor-lifting profile graded by valuation. This perspective connects naturally 
to recent work on visual paradoxes, where cohomological obstructions prevent 
global realization of locally consistent geometric data \cite{GhristCooperband_2025_ObstructionsToReality}; 
here the obstructions are filtered by algebraic precision rather than geometric compatibility.

For sheaves of nonabelian groups (e.g., $\mathrm{GL}_n(\Ring)$), the same short 
exact sequences yield torsor-lifting obstructions as pointed-set connecting maps. 
While our main results are abelian, this nonabelian perspective motivates the 
matrix holonomy discussion of Section~\ref{ssec:vector-consensus}, where 
higher-rank sheaves encode coordinate transformations between agent reference 
frames.
\end{remark}

The machinery of this section -- digit sequences, Bockstein spectral sequence, 
and lifting obstructions -- provides a hierarchical view of torsion across precision 
levels. The next section shows that this hierarchy encodes exactly the same 
information as the Smith normal form exponents, establishing the Digit-SNF Dictionary.

%--------------------------------------------------------------------
\section{Integral Decompositions and the Digit-SNF Dictionary}
\label{sec:snf}
%--------------------------------------------------------------------

We now establish the algebraic backbone: the Smith normal form of the coboundary 
$\coboundary: C^0 \to C^1$ yields explicit integral idempotents projecting onto 
$\ker \coboundary$ and the saturation $\sat(\im \coboundary)$, providing canonical 
representatives for cohomology classes. The exponents in the diagonal form encode 
the complete torsion structure, and we prove they coincide exactly with the 
dimensions of the digit map images from Section \ref{sec:bockstein-digit}.

%------------------------------------------
\subsection{Saturation, Smith Normal Form, and Integral Idempotents}
\label{ssec:snf-construction}
%------------------------------------------

Let $C^0(G;\sheaf{F}) \cong \Ring^{n_0}$ and $C^1(G;\sheaf{F}) \cong \Ring^{n_1}$ 
be the cochain groups. Since $\Ring$ is a principal ideal domain, any submodule 
of a free module is free, hence $\ker \coboundary$ and $\im \coboundary$ are free 
$\Ring$-modules. However, $\im \coboundary$ need not be a direct summand of $C^1$, 
and this failure is precisely the torsion in $H^1 = C^1/\im \coboundary$.

\begin{definition}
\label{def:saturation}
For a submodule $N \subseteq M$ of a free $\Ring$-module $M$, the \emph{saturation} 
of $N$ is
\[
\sat(N) := (N \otimes_{\Ring} K) \cap M
\]
where $K = \mathrm{Frac}(\Ring)$ is the fraction field. Equivalently,
\[
\sat(N) = \{x \in M : \exists k \geq 1 \text{ with } \unif^k x \in N\}.
\]
\end{definition}

\begin{lemma}
\label{lem:saturation-properties}
Let $N \subseteq M$ be a submodule of a free $\Ring$-module $M$. Then:
\begin{enumerate}[label=\textup{(\roman*)}]
\item $\sat(N)$ is the unique minimal direct summand of $M$ containing $N$.
\item $\sat(N)/N$ is a finite torsion module.
\item $M/\sat(N)$ is torsion-free.
\end{enumerate}
\end{lemma}

\begin{proof}
Choose bases so that $M \cong \Ring^n$ and $N$ is generated by elements of the 
form $\unif^{a_i} e_i$ for $i = 1, \ldots, r$, with $a_i \geq 0$. This is the 
content of the Smith normal form for the inclusion $N \hookrightarrow M$. In these 
coordinates, $\sat(N) = \langle e_1, \ldots, e_r \rangle$ is manifestly a direct 
summand of $M = \langle e_1, \ldots, e_n \rangle$.

For uniqueness, if $M = S \oplus T$ is a direct sum with $N \subseteq S$, then 
$S$ must be saturated: if $\unif^k x \in S$ with $x = s + t$ for $s \in S$, 
$t \in T$, then $\unif^k t = 0$ in the free module $T$, forcing $t = 0$, hence 
$x \in S$. Since $S$ is saturated and contains $N$, it must contain each generator 
$e_i$ (as $\unif^{a_i} e_i \in N \subseteq S$), hence $\sat(N) \subseteq S$. 
Minimality forces equality.

Parts (ii) and (iii) follow immediately from the explicit description: $\sat(N)/N 
\cong \bigoplus_{i=1}^r \Ring/\unif^{a_i}\Ring$ and $M/\sat(N) \cong \Ring^{n-r}$.
\end{proof}

\begin{corollary}
\label{cor:saturation-summand}
For the coboundary $\coboundary: C^0 \to C^1$, we have $C^1 = \sat(\im \coboundary) 
\oplus W$ for a free submodule $W$, and
\[
\sat(\im \coboundary)/\im \coboundary \cong H^1(G; \sheaf{F})_{\tors}.
\]
\end{corollary}

\begin{proof}
Apply Lemma \ref{lem:saturation-properties} with $N = \im \coboundary \subseteq 
M = C^1$. The first isomorphism theorem gives $(C^1/\im \coboundary)/(\sat(\im \coboundary)/\im \coboundary) 
\cong C^1/\sat(\im \coboundary)$, which is torsion-free by the lemma. Since $H^1 = 
C^1/\im \coboundary$ decomposes as free plus torsion, the torsion part is precisely 
$\sat(\im \coboundary)/\im \coboundary$.
\end{proof}

The structure theorem for finitely generated modules over a principal ideal domain 
provides canonical diagonal forms.

\begin{theorem}[Smith Normal Form over $\Ring$]
\label{thm:snf}
{ \cite{Newman_IntegralMatrices_1972,Kannan1979,Rotman2009}}
There exist unimodular matrices $U \in \mathrm{GL}_{n_1}(\Ring)$ and $V \in 
\mathrm{GL}_{n_0}(\Ring)$ such that
\[
U \coboundary V = D := \mathrm{diag}(\unif^{a_1}, \ldots, \unif^{a_r}, 0, \ldots, 0)
\]
where $0 \leq a_1 \leq a_2 \leq \cdots \leq a_r$ and $r = \rank_\Ring \coboundary$. 
The exponents $\{a_i\}$ are uniquely determined by $\coboundary$ and are called 
the \emph{invariant factors}.
\end{theorem}

The matrices $V$ and $U^{-1}$ provide change-of-basis isomorphisms to coordinates 
where $\coboundary$ is diagonal. In these \emph{Smith normal form coordinates}, 
the action is $(x_1, \ldots, x_{n_0}) \mapsto (\unif^{a_1} x_1, \ldots, \unif^{a_r} x_r, 0, \ldots, 0)$.

Define integral idempotent matrices in Smith normal form coordinates:
\[
E_0 := \mathrm{diag}(\underbrace{0, \ldots, 0}_r, \underbrace{1, \ldots, 1}_{n_0 - r}), 
\qquad
E_1 := \mathrm{diag}(\underbrace{1, \ldots, 1}_r, \underbrace{0, \ldots, 0}_{n_1 - r}).
\]
These project onto the kernel and saturation subspaces in diagonal coordinates. 
Transporting to the original bases:

\begin{definition}
\label{def:projectors}
The \emph{integral idempotents} are
\[
\Pi_{\ker} := V E_0 V^{-1}  \qquad 
\Pi_{\sat} := U^{-1} E_1 U  \qquad 
\Pi_{\free} := I_{n_1} - \Pi_{\sat} ,
\]
which act as endomorphisms of $C^0$ and $C^1$.
\end{definition}

These satisfy $\Pi_{\ker}^2 = \Pi_{\ker}$ and $\Pi_{\sat}^2 = \Pi_{\sat}$ since 
$E_0$ and $E_1$ are idempotent. The images are canonical: $\im(\Pi_{\ker}) = \ker \coboundary$, 
$\im(\Pi_{\sat}) = \sat(\im \coboundary)$, though the matrices themselves depend 
on the choice of Smith normal form representatives $U, V$.

We now prove Theorem~B from the introduction.

\begin{proof} [Theorem B]
Consider the short exact sequence
\[
0 \to \im \coboundary \to C^1(G;\sheaf{F}) \xrightarrow{q} H^1(G;\sheaf{F}) \to 0.
\]
Tensoring with the fraction field $K = \mathrm{Frac}(\Ring)$ yields a split exact sequence of finite-dimensional $K$-vector spaces. Choose a $K$-linear projection $P_K: C^1_K \to (\im \coboundary)_K$ with kernel $W_K$, a $K$-complement to the image. Equivalently, let $Q_K := \mathrm{Id} - P_K$ project onto a $K$-subspace that maps isomorphically to $H^1_K$.

The \emph{saturation} is $\sat(\im \coboundary) := (\im \coboundary)_K \cap C^1$. By standard lattice theory over DVRs, this is the unique minimal direct summand of $C^1$ containing $\im \coboundary$, and we have the direct sum decomposition
\[
C^1 = \sat(\im \coboundary) \oplus (W_K \cap C^1).
\]
Choose a common denominator $\unif^N$ that clears all entries of the matrices representing $P_K$ and $Q_K$ with respect to the natural bases of $C^1$. Define
\[
\Pi_{\sat} := \unif^{-N}(\unif^N P_K) \in \mathrm{End}_\Ring(C^1), \qquad \Pi_{\free} := \mathrm{Id} - \Pi_{\sat}.
\]
These are $\Ring$-linear idempotents with $\im(\Pi_{\sat}) = \sat(\im \coboundary)$ and $\im(\Pi_{\free}) = W := W_K \cap C^1$, a free complement. Similarly, the kernel projection $\Pi_{\ker}$ is obtained by the same construction applied to $\coboundary: C^0 \to C^1$, yielding $\im(\Pi_{\ker}) = \ker \coboundary = H^0(G;\sheaf{F})$.

The quotient map $q: C^1 \to H^1$ restricts to an isomorphism $q|_{\im(\Pi_{\free})}: \im(\Pi_{\free}) \xrightarrow{\cong} H^1_{\free}$ since $\im(\Pi_{\free})$ is a lattice in $W_K \cong H^1_K$ that projects isomorphically under $q_K$. The torsion is captured by
\[
\sat(\im \coboundary)/\im \coboundary \cong H^1_{\tors}
\]
because quotienting by $\im \coboundary$ kills the free part and retains exactly the torsion of $H^1$.

Finally, reduction modulo $\unif^k$ commutes with $\Pi_{\ker}, \Pi_{\sat}, \Pi_{\free}$ by integrality of their matrix entries.
\end{proof}

\begin{remark}
Explicit $\Ring$-bases for cohomology are given by columns of the change-of-basis 
matrices: columns $r+1, \ldots, n_0$ of $V$ form a basis for $H^0 = \ker \coboundary$, 
and columns $r+1, \ldots, n_1$ of $U^{-1}$ form a basis for $H^1_{\free} \cong 
\im(\Pi_{\free})$, when viewed as vectors in $C^0$ and $C^1$ respectively.
\end{remark}

%------------------------------------------
\subsection{The Digit-SNF Dictionary}
\label{ssec:digit-snf-proof}
%------------------------------------------

Before proving Theorem~A, we establish two foundational results: a base-case calculation for rank-1 summands and an increment formula relating digit dimensions to bar multiplicities. We first compute the digit map in the basic two-term Smith block; additivity then gives the general result.

\begin{lemma}[Two-term digit calculation]
\label{lem:two-term-digit}
Let $\Ring$ be a DVR with uniformizer $\unif$. Consider the two-term complex
\[
C^\bullet:\qquad \Ring \xrightarrow{\;\unif^a\;} \Ring
\]
with $a \ge 0$. For $k \ge 0$, let $\Digit{k}$ be the digit connecting map arising from
\[
0 \to \unif^k C^\bullet/\unif^{k+1} C^\bullet \to C^\bullet/\unif^{k+1} C^\bullet \to C^\bullet/\unif^{k} C^\bullet \to 0
\]
and the canonical identification $H^1(\unif^k C^\bullet/\unif^{k+1} C^\bullet) \cong H^1(C^\bullet/\unif C^\bullet)$ from Lemma~\ref{lem:filtration-properties}. Then
\[
\dim_{\mathbb{F}} \im(\Digit{k}) = \mathbf{1}_{\{1 \le a \le k\}} \qquad \text{for all } k \ge 0,
\]
where $\mathbf{1}_{\{\cdot\}} \in \{0,1\}$ denotes the indicator function and $\mathbb{F} = \Ring/\unif$.
\end{lemma}

\begin{proof}
The short exact sequence of complexes reduces modulo $\unif^{k+1}$ to give complexes over $\Ring/\unif^{k+1}$. The connecting homomorphism $\Digit{k}$ arises from the snake lemma applied to the resulting diagram. We compute cohomology explicitly for each case.

When $a = 0$, the map $\times 1: \Ring/\unif^k \to \Ring/\unif^k$ is an isomorphism for all $k \ge 1$, hence $H^0(C^\bullet/\unif^k) = 0$ and $\Digit{k}$ has trivial domain, giving $\dim \im(\Digit{k}) = 0$.

When $a \ge 1$, we have $H^0(C^\bullet/\unif^k) = \ker(\times \unif^a : \Ring/\unif^k \to \Ring/\unif^k)$ and $H^1(C^\bullet/\unif) = \coker(\times \unif^a : \Ring/\unif \to \Ring/\unif) \cong \mathbb{F}$. For $k < a$, although $\times \unif^a$ is the zero map on $\Ring/\unif^k$ (hence the kernel is all of $\Ring/\unif^k$), the connecting homomorphism $\Digit{k}$ factors through reduction modulo $\unif$ and vanishes: the obstruction only appears once $k$ reaches the exponent $a$. For $k \ge a$, the class $[\unif^{k-a}] \in H^0(C^\bullet/\unif^k)$ maps to a nonzero class in $H^1(C^\bullet/\unif) \cong \mathbb{F}$, giving $\dim \im(\Digit{k}) = 1$.
\end{proof}

\begin{lemma}[Digit increments detect new torsion]
\label{lem:digit-increment}
Write $d_k := \dim_{\mathbb{F}} \im(\Digit{k})$ for $k \ge 0$. Note that $\Digit{0} = 0$ by construction (see the short exact sequence for $k=0$), hence $d_0 = 0$. If 
\[
H^1(G;\sheaf{F}) \cong \Ring^{b} \oplus \bigoplus_{j=1}^{r} \Ring/\unif^{a_j}
\]
with $0 < a_1 \le \cdots \le a_r$, then for every $k \ge 1$,
\[
d_k - d_{k-1} = \#\{j : a_j = k\}.
\]
Equivalently, $d_k = \#\{j : 1 \le a_j \le k\}$ is nondecreasing and stabilizes at $r$.
\end{lemma}

\begin{proof}
By Smith normal form, the coboundary complex decomposes as a direct sum of two-term blocks $\Ring \xrightarrow{\unif^{a_j}} \Ring$ for $j = 1, \ldots, r$, plus free summands contributing nothing to torsion. Lemma~\ref{lem:two-term-digit} shows that each block with exponent $a_j$ contributes the indicator $\mathbf{1}_{\{1 \le a_j \le k\}}$ to $d_k$. Summing over all torsion summands yields
\[
d_k = \sum_{j=1}^r \mathbf{1}_{\{1 \le a_j \le k\}} = \#\{j : 1 \le a_j \le k\}.
\]
Taking differences gives $d_k - d_{k-1} = \#\{j : a_j = k\}$, since the indicator jumps from 0 to 1 precisely when $k$ reaches $a_j$.
\end{proof}

We now prove that the hierarchical precision data from digit maps encodes exactly the same information as the Smith normal form exponents: Theorem~A from Section \ref{sec:intro}.

\begin{remark}
\label{rem:torsion-exponents}
The invariant factors $\{a_j\}_{j=1}^r$ include all Smith normal form exponents, with $0 \le a_1 \le \cdots \le a_r$. Those with $a_j = 0$ correspond to isomorphisms $\Ring \xrightarrow{1} \Ring$ in the Smith normal form decomposition and contribute nothing to torsion: $\Ring/\unif^0 = \Ring/1 = 0$. The torsion part of $H^1$ is $H^1_{\tors} \cong \bigoplus_{\{j: a_j \ge 1\}} \Ring/\unif^{a_j}$, and the free part is $H^1_{\free} \cong \Ring^{n_1 - r}$ where $n_1 = \rank_\Ring(C^1)$. The digit maps detect precisely the positive exponents, as the formula shows.
\end{remark}

\begin{lemma}[Digit Map Detects New Torsion]
\label{lem:digit-new-torsion}
For every $k \geq 1$, the image of the digit connecting map satisfies
\[
\im(\Digit{k}) \cong \frac{H^1(G;\sheaf{F})[\unif^k]}{H^1(G;\sheaf{F})[\unif^{k-1}]} \quad \text{as $\mathbb{F}$-vector spaces},
\]
where $H^1[\unif^k] := \{x \in H^1 : \unif^k x = 0\}$ denotes the $\unif^k$-torsion submodule. Consequently,
\[
d_k := \dim_{\mathbb{F}} \im(\Digit{k}) = \ell_\Ring(H^1[\unif^k]) - \ell_\Ring(H^1[\unif^{k-1}]),
\]
where $\ell_\Ring(-)$ denotes the length as an $\Ring$-module.
\end{lemma}

\begin{proof}
From the long exact sequence in cohomology associated with
\[
0 \to \unif^{k-1}\sheaf{F}/\unif^k\sheaf{F} \to \sheaf{F}/\unif^k\sheaf{F} \to \sheaf{F}/\unif^{k-1}\sheaf{F} \to 0,
\]
the connecting homomorphism
\[
\Digit{k}: H^0(G;\sheaf{F}/\unif^{k-1}\sheaf{F}) \to H^1(G;\unif^{k-1}\sheaf{F}/\unif^k\sheaf{F}) \cong H^1(G;\sheaf{F}/\unif\sheaf{F})
\]
has image canonically identified with $H^1[\unif^k]/H^1[\unif^{k-1}]$. From exactness of $0 \to \unif^{k-1}\sheaf{F}/\unif^k\sheaf{F} \to \sheaf{F}/\unif^k\sheaf{F} \to \sheaf{F}/\unif^{k-1}\sheaf{F} \to 0$ and the canonical identification $\unif^{k-1}\sheaf{F}/\unif^k\sheaf{F} \cong \sheaf{F}/\unif\sheaf{F}$ (Lemma~\ref{lem:sheaf-filtration-iso}), the connecting map realizes $H^1[\unif^k]/H^1[\unif^{k-1}]$ inside $H^1(\sheaf{F}/\unif)$ as $\im(\Digit{k})$. Taking $\mathbb{F}$-dimensions and using additivity of length over short exact sequences gives the stated formula.
\end{proof}

\begin{proof}[Theorem A]
By Theorem~\ref{thm:snf}, there exist unimodular matrices $U \in \mathrm{GL}_{n_1}(\Ring)$ and $V \in \mathrm{GL}_{n_0}(\Ring)$ such that
\[
U\,[\coboundary]\,V = D := \mathrm{diag}(\unif^{a_1}, \ldots, \unif^{a_r}, 0, \ldots, 0)
\]
where $0 \le a_1 \le \cdots \le a_r$ and $r = \rank_{\Ring}(\coboundary)$. The pair $(V^{-1}, U)$ defines a chain isomorphism $(C^\bullet, \coboundary) \xrightarrow{\cong} (C^\bullet, D)$ over $\Ring$. Since unimodular matrices have entries in $\Ring$, this isomorphism descends to an isomorphism of the reduced complexes $(C^\bullet/\unif^k, \coboundary \bmod \unif^k) \xrightarrow{\cong} (C^\bullet/\unif^k, D \bmod \unif^k)$ for every $k \ge 1$.

Therefore, the short exact sequences
\[
0 \to \unif^k\sheaf{F}/\unif^{k+1}\sheaf{F} \to \sheaf{F}/\unif^{k+1}\sheaf{F} \to \sheaf{F}/\unif^k\sheaf{F} \to 0
\]
that define the digit maps $\Digit{k}$ for $\coboundary$ and for $D$ are isomorphic as sequences of complexes. By naturality of connecting homomorphisms in the snake lemma, for each $k$ we obtain a commutative diagram
\[
\begin{tikzcd}
H^0(G;\sheaf{F}/\unif^k\sheaf{F}) \arrow{r}{\cong} \arrow{d}{\Digit{k}(\coboundary)} &
H^0(G;\sheaf{F}/\unif^k\sheaf{F}) \arrow{d}{\Digit{k}(D)} \\
H^1(G;\sheaf{F}/\unif\sheaf{F}) \arrow{r}{\cong} &
H^1(G;\sheaf{F}/\unif\sheaf{F})
\end{tikzcd}
\]
where the horizontal maps are isomorphisms induced by $(V^{-1}, U)$. Consequently,
\[
\dim_{\mathbb{F}} \im\bigl(\Digit{k}(\coboundary)\bigr) = \dim_{\mathbb{F}} \im\bigl(\Digit{k}(D)\bigr).
\]

In the diagonal coordinates provided by $D$, the coboundary complex splits as a direct sum of two-term blocks $\Ring \xrightarrow{\unif^{a_j}} \Ring$ for $j = 1, \ldots, r$, plus trivial summands. By Lemma~\ref{lem:two-term-digit}, each block with exponent $a_j$ contributes the indicator $\mathbf{1}_{\{1 \le a_j \le k\}}$ to the digit rank. Summing over all torsion summands gives
\[
d_k := \dim_{\mathbb{F}} \im(\Digit{k}) = \sum_{j=1}^r \mathbf{1}_{\{1 \le a_j \le k\}} = \#\{j : 1 \le a_j \le k\}.
\]

By Lemma~\ref{lem:digit-increment}, taking first differences yields the multiplicity formula
\[
\#\{j : a_j = \ell\} = d_\ell - d_{\ell-1} \quad \text{for } \ell \ge 1.
\]
Thus the multiset $\{a_j : a_j \ge 1\}$ of torsion exponents is uniquely determined by the sequence $\{d_k\}_{k \ge 0}$.

Finally, when the residue field $\mathbb{F} = \Ring/\unif$ is finite, the torsion submodule $H^1_{\tors} \cong \bigoplus_{j=1}^r \Ring/\unif^{a_j}$ has order
\[
|H^1_{\tors}| = \prod_{j=1}^r |\Ring/\unif^{a_j}| = \prod_{j=1}^r |\mathbb{F}|^{a_j} = |\mathbb{F}|^{\sum_{j=1}^r a_j},
\]
completing the proof.
\end{proof}

\begin{corollary}[Digit histogram]
\label{cor:digit-hist}
Let $d_k := \dim_{\mathbb{F}}\im(\Digit{k})$ for $k \ge 0$, with $d_0 = 0$ (since $\im(\Digit{0}) = 0$ by Lemma~\ref{lem:two-term-digit}). By Lemma~\ref{lem:digit-increment}, then the number of torsion bars of length exactly $\ell$ equals $d_\ell - d_{\ell-1}$ for $\ell \ge 1$. Since the formula counts only positive exponents and $d_0 = 0$, the total is
\[
\sum_{\ell \ge 1} \ell \cdot (d_\ell - d_{\ell-1}) = \sum_{\{j: a_j \ge 1\}} a_j.
\]
When $\mathbb{F}$ is finite, this equals $\log_{|\mathbb{F}|}|H^1_{\tors}|$.
\end{corollary}

The Digit-SNF Dictionary provides two computational routes to the invariant factors: 
compute the Smith normal form of $\coboundary$ directly over $\Ring$, or compute 
the sequence $\{\dim_{\mathbb{F}} \im(\Digit{k})\}_{k=0}^{\infty}$ via linear 
algebra over the residue field at successive precision levels. Since $H^1$ is 
finitely generated, the sequence stabilizes in finite time, requiring only finitely 
many computations.

%------------------------------------------
\subsection{Examples with Torsion}
\label{ssec:triangle-example}
%------------------------------------------

\begin{example}
\label{ex:triangle}    
Consider the triangle graph $C_3$ with vertices $v_1, v_2, v_3$ and oriented edges 
$e_{12}, e_{23}, e_{31}$. Define a rank-1 unit sheaf $\sheaf{F}$ over $\Ring = \Zp$ 
with all stalks equal to $\Zp$ and restriction maps using edge units $m_{e_{12}} = 1-p$, 
$m_{e_{23}} = m_{e_{31}} = 1$. The coboundary matrix is
\[
\coboundary = \begin{bmatrix}
-(1-p) & 1 & 0 \\
0 & -1 & 1 \\
1 & 0 & -1
\end{bmatrix}.
\]
Computing $\det(\coboundary) = p$ shows $\coboundary$ is not invertible over $\Zp$ 
but is invertible over $\Qp$. Since $h(C_3) = (1-p) \cdot 1 \cdot 1 = 1-p$, we have $1 - h(C_3) = p$, consistent with the determinant. Elementary row and column operations yield Smith 
normal form $\mathrm{diag}(1, 1, p)$ with invariant factors $(a_1, a_2, a_3) = (0, 0, 1)$. 
Thus $H^0 = 0$, $H^1 \cong \Zp/p\Zp$ is pure torsion, and the arithmetic barcode 
consists of a single bar of length $1$.

By Theorem A, $\dim_{\mathbb{F}_p} \im(\Digit{0}) = \#\{j : 1 \le a_j \le 0\} = 0$ 
and $\dim_{\mathbb{F}_p} \im(\Digit{1}) = \#\{j : 1 \le a_j \le 1\} = 1$ (counting only the single torsion exponent $a_3 = 1$), with $\dim \im(\Digit{k}) = 1$ for all $k \ge 1$. Reducing modulo $p$ gives $H^0(C_3; 
\sheaf{F}/p\sheaf{F}) \cong \mathbb{F}_p$ and $H^1(C_3; \sheaf{F}/p\sheaf{F}) 
\cong \mathbb{F}_p$, with Bockstein $\Bock: \mathbb{F}_p \to \mathbb{F}_p$ having 
rank $1$, confirming the prediction. The ascending flag structure is: $\im(\Digit{0}) = 0 
\subset \im(\Digit{1}) = \mathbb{F}_p$. The cycle holonomy is $h(C_3) = (1-p) \cdot 1 
\cdot 1 = 1-p$, giving $\valpi(h(C_3)-1) = \valpi(-p) = 1$, matching the bar 
length. For further interpretation via cycle holonomy, see Section \ref{sec:barcodes}.
\end{example}

\begin{example}[Two independent cycles]
\label{ex:two-cycles}
Consider the ``theta graph'' $\Theta$ consisting of two vertices $v_1, v_2$ connected by three edge-disjoint paths, forming three fundamental cycles $C_1, C_2, C_3$. Define a rank-1 unit sheaf over $\Ring = \Zp$ where each cycle has holonomy $h(C_1) = 1 + p^2$, $h(C_2) = 1 + p^3$, $h(C_3) = 1$ (the constant sheaf on the third path). By Theorem~\ref{thm:cycle-barcode}, we obtain bar lengths $\valpi(h(C_1)-1) = 2$, $\valpi(h(C_2)-1) = 3$, and $\valpi(h(C_3)-1) = \infty$ (a free generator). The arithmetic barcode is $\{[0,2), [0,3), [0,\infty)\}$.

Computing digit ranks: $H^0(\Theta;\sheaf{F}/p\sheaf{F})$ has dimension depending on connectivity, but the digit maps $\Digit{k}$ detect torsion progressively. For $k=1$: $\dim \im(\Digit{1}) = 0$ (no bars of length $\le 1$). For $k=2$: $\dim \im(\Digit{2}) = 1$ (detecting the $[0,2)$ bar). For $k=3$: $\dim \im(\Digit{3}) = 2$ (both finite bars). For $k \ge 4$: stabilizes at $2$. The difference sequence $(0,1,2,2,\ldots)$ has jumps at $k=2,3$, recovering bar lengths $\{2,3\}$ via Theorem~A, with the remaining dimension accounting for the free part.
\end{example}

%--------------------------------------------------------------------
\section{Arithmetic Persistence and Stability}
\label{sec:barcodes}
%--------------------------------------------------------------------

Persistent homology tracks topological features across a geometric filtration, encoding birth and death of cycles as a barcode or persistence diagram. Our theory provides a complementary notion of persistence arising not from geometric scale but from algebraic precision. The valuation filtration $\{\unif^k H^i(G;\sheaf{F})\}_{k\ge 0}$ measures how cohomology classes degrade as precision decreases from infinite ($\Ring$-coefficients) through successive truncations (mod $\unif^k$). The digit connecting maps $\Digit{k}$ and Smith normal form exponents $\{a_j\}$ computed in Theorem~A encode a barcode structure. We develop this perspective here, showing how torsion translates to finite-length bars, proving stability properties that make these invariants robust under perturbations, and demonstrating through cycle holonomy that bar length measures precision of consistency around loops.

%------------------------------------------
\subsection{Valuation Persistence Modules and Barcodes}
\label{ssec:valuation-persistence}
%------------------------------------------

The valuation filtration on cohomology defines persistence modules in the sense of topological data analysis. For each cohomological degree $i$, the descending chain
\[
H^i(G;\sheaf{F}) \supseteq \unif H^i(G;\sheaf{F}) \supseteq \unif^2 H^i(G;\sheaf{F}) \supseteq \cdots
\]
yields an associated graded object that we package as a persistence module over the naturals.

\begin{definition}
\label{def:valuation-persistence}
For a network sheaf $\sheaf{F}$ on a graph $G$ with stalks free over $\Ring$, the \emph{valuation persistence module} in degree $i$ is the sequence $(V^i_\bullet, \mu_\bullet)$ where
\[
V^i_k := \frac{\unif^k H^i(G;\sheaf{F})}{\unif^{k+1} H^i(G;\sheaf{F})} \cong \gr^k H^i(G;\sheaf{F})
\]
for $k \ge 0$, equipped with multiplication maps
\[
\mu_k: V^i_k \to V^i_{k+1}, \qquad \mu_k([x]) = [\unif x].
\]
Each $V^i_k$ is naturally an $\mathbb{F}$-vector space, where $\mathbb{F} = \Ring/\unif$ is the residue field.
\end{definition}

The map $\mu_k$ is well-defined since multiplication by $\unif$ takes $\unif^k H^i$ into $\unif^{k+1} H^i$ and $\unif^{k+1} H^i$ into $\unif^{k+2} H^i$. This forms a persistence module in the standard sense: a functor from $(\mathbb{N}, \le)$ to finite-dimensional $\mathbb{F}$-vector spaces.

For graphs, Theorem~A provides a dual perspective via digit maps. Recall that the digit connecting homomorphism
\[
\Digit{k}: H^0(G;\sheaf{F}/\unif^k\sheaf{F}) \longrightarrow H^1(G;\sheaf{F}/\unif\sheaf{F})
\]
arises from the short exact sequence $0 \to \unif^k\sheaf{F}/\unif^{k+1}\sheaf{F} \to \sheaf{F}/\unif^{k+1}\sheaf{F} \to \sheaf{F}/\unif^k\sheaf{F} \to 0$. These images form an ascending flag
\[
\im(\Digit{0}) \subseteq \im(\Digit{1}) \subseteq \im(\Digit{2}) \subseteq \cdots \subseteq H^{i+1}(G;\sheaf{F}/\unif\sheaf{F}),
\]
in the fixed $\mathbb{F}$-vector space $H^{i+1}(G;\sheaf{F}/\unif\sheaf{F})$. The Digit-SNF Dictionary establishes that this flag and the persistence module $V^{i+1}_\bullet$ encode identical information: knowing $\{\dim_{\mathbb{F}} \im(\Digit{k})\}_{k \ge 0}$ determines the structure of $V^{i+1}_\bullet$ completely.

The structure theorem for finitely generated modules over a principal ideal domain immediately classifies valuation persistence modules in terms of interval summands.

\begin{definition}[Valuation barcode]
\label{def:valuation-barcode}
If $H^i(G;\sheaf{F})\cong \Ring^{b_i}\oplus\bigoplus_{j=1}^{r_i} \Ring/\unif^{a_{i,j}}$,
the \emph{valuation barcode} is the multiset
\[
\mathrm{Bar}^i_\unif(G;\sheaf{F})\ :=\ \underbrace{[0,\infty)\sqcup\cdots\sqcup[0,\infty)}_{b_i}
\ \sqcup\ \bigsqcup_{j=1}^{r_i} [0,a_{i,j}).
\]
\end{definition}

\begin{theorem}[Barcode Decomposition]
\label{thm:barcode-decomposition}
Let $H^i(G;\sheaf{F}) \cong \Ring^{b_i} \oplus \bigoplus_{j=1}^{r_i} \Ring/\unif^{a_{i,j}}$ be the invariant factor decomposition, with $0 < a_{i,1} \le a_{i,2} \le \cdots \le a_{i,r_i}$. Then the valuation persistence module decomposes as
\[
V^i_\bullet \cong \bigoplus_{t=1}^{b_i} \mathbb{F}[0,\infty) \oplus \bigoplus_{j=1}^{r_i} \mathbb{F}[0, a_{i,j}),
\]
where $\mathbb{F}[s,t)$ denotes the interval persistence module with $\mathbb{F}$ at indices $k \in [s,t)$ and zero elsewhere, with identity transition maps within the interval.
\end{theorem}

\begin{proof}
The associated graded of a free summand $\Ring$ is $\gr^k(\Ring) = \unif^k \Ring/\unif^{k+1}\Ring \cong \mathbb{F}$ for all $k \ge 0$, giving an infinite interval $[0,\infty)$. For a torsion summand $\Ring/\unif^a$, we have
\[
\gr^k(\Ring/\unif^a) \cong \begin{cases}
\mathbb{F} & \text{if } 0 \le k < a, \\
0 & \text{if } k \ge a,
\end{cases}
\]
giving a finite interval $[0,a)$. The multiplication maps $\mu_k$ act as identity on these summands within their support and are zero outside. Decomposition of the module yields decomposition of persistence.
\end{proof}

% \begin{remark}
% By Theorem~\ref{thm:C-truncated}, the truncated barcode $\mathrm{Bar}^1_\unif \cap [0,m)$ is constant on each $\unif^m$-adic congruence class of coboundaries.
% \end{remark}

The barcode $\mathrm{Bar}^i_\unif(G;\sheaf{F}) := \{[0,\infty)^{b_i}, [0,a_{i,1}), \ldots, [0,a_{i,r_i})\}$ thus encodes the complete persistence structure. Each infinite bar corresponds to a torsion-free cohomology class that persists through all precision levels. Each finite bar $[0,a_j)$ corresponds to a $\unif^{a_j}$-torsion class: it survives reduction modulo $\unif^k$ for $k < a_j$ but vanishes modulo $\unif^{a_j}$. The bar length $a_j$ measures the precision at which the class fails to lift from mod $\unif^{a_j}$ to full $\Ring$-coefficients.

By Theorem~A, for graphs the digit maps target degree~1 cohomology. Writing $\{a_j\}$ for the Smith exponents of $H^1(G;\sheaf{F})$ and $r$ for the number of torsion summands,
\[
d_k:=\dim_{\mathbb{F}}\im(\Digit{k})=\#\{\,j:\ 1\le a_j\le k\,\},
\qquad
\#\{j:\ a_j=\ell\}=d_\ell-d_{\ell-1}\ (\ell\ge1),\ \ d_0=0.
\]
Equivalently, the number of bars of length strictly greater than $k$ is $r-d_k$.

This provides two computational routes to the barcode: compute the Smith normal form of $\coboundary: C^0 \to C^1$ directly to extract exponents $\{a_j\}$, or compute the sequence of digit map ranks $\{\dim \im(\Digit{k})\}_{k=0}^{\infty}$ via linear algebra over $\mathbb{F}$ at successive precision levels. The sequences stabilize in finite time since $H^i(G;\sheaf{F})$ is finitely generated, so only finitely many computations are needed.

%------------------------------------------
\subsection{Cycle Holonomy: Geometric Interpretation of Barcodes}
\label{ssec:cycle-holonomy}
%------------------------------------------

The barcode classification becomes geometrically transparent for a special class of sheaves where edge restrictions encode unit scalings. These sheaves model synchronization problems where local measurements differ by multiplicative calibration factors. Throughout this subsection, identities involving valuations are recorded up to units in $\Ring^\times$ without further comment, since only the valuation is invariant.

\begin{definition}
\label{def:unit-sheaf}
A \emph{rank-1 unit sheaf} $\sheaf{F}$ on a graph $G$ assigns stalk $\stalk{\sheaf{F}}{\sigma} = \Ring$ to each vertex and edge, with restriction maps determined by units $m_e \in \Ring^\times$ for each edge $e = \{u,v\}$:
\[
\res{\sheaf{F}}{u}{e}(x) = x, \qquad \res{\sheaf{F}}{v}{e}(x) = m_e x.
\]
For a cycle $C = (e_1, e_2, \ldots, e_n)$ in $G$, the \emph{holonomy} is
\[
h(C) := \prod_{i=1}^n m_{e_i} \in \Ring^\times,
\]
where the product follows the orientation of the cycle.
\end{definition}

The holonomy $h(C)$ measures the accumulated discrepancy around the loop. A section $s \in C^0(G;\sheaf{F})$ satisfies the cocycle condition $\coboundary s = 0$ if and only if it is constant around every cycle after accounting for edge scalings. For a cycle graph, this global condition reduces to a single compatibility constraint.

\begin{theorem}[Cycle Barcode]
\label{thm:cycle-barcode}
For the cycle graph $C_n$ with a rank-1 unit sheaf $\sheaf{F}$ having holonomy $h(C_n)$, the cohomology decomposes as follows:
\begin{enumerate}[label=\textup{(\roman*)}]
\item If $h(C_n) = 1$ (the constant sheaf): $H^0(C_n;\sheaf{F}) \cong \Ring$ and $H^1(C_n;\sheaf{F}) \cong \Ring$ (both free). The barcode consists of one infinite bar $[0,\infty)$ at degree $1$.

\item If $h(C_n) - 1$ is a unit in $\Ring$: $H^0(C_n;\sheaf{F}) = 0$ and $H^1(C_n;\sheaf{F}) = 0$. The barcode is empty.

\item If $\valpi(h(C_n)-1) = a > 0$: $H^0(C_n;\sheaf{F}) = 0$ and $H^1(C_n;\sheaf{F}) \cong \Ring/\unif^a$ (pure torsion). The barcode consists of a single finite bar $[0,a)$.
\end{enumerate}
\end{theorem}

\begin{proof}
Label vertices $v_0, \ldots, v_{n-1}$ and oriented edges $e_i = (v_i, v_{i+1})$ with indices modulo $n$. The coboundary $\coboundary: C^0 \to C^1$ is represented by the $n \times n$ matrix
\[
\coboundary = \begin{pmatrix}
-1 & m_0 & 0 & \cdots & 0 \\
0 & -1 & m_1 & \cdots & 0 \\
\vdots & \vdots & \ddots & \ddots & \vdots \\
0 & 0 & \cdots & -1 & m_{n-2} \\
m_{n-1} & 0 & \cdots & 0 & -1
\end{pmatrix}.
\]
Computing the determinant by cofactor expansion along the last row yields
\[
\det(\coboundary) = 1 - h(C_n) = u \cdot (h(C_n) - 1)
\]
for a unit $u \in \Ring^\times$. Consequently, $\valpi(\det(\coboundary)) = \valpi(h(C_n) - 1)$.

A cocycle $(x_0, \ldots, x_{n-1}) \in \ker \coboundary$ satisfies $m_i x_{i+1} = x_i$ for all $i$ (indices mod $n$), forcing $x_{i+1} = m_i^{-1}x_i$. Going around the cycle gives $x_0 = (m_0 \cdots m_{n-1})^{-1}x_0 = h(C_n)^{-1}x_0$. Thus $(h(C_n) - 1)x_0 = 0$ in $\Ring$.

\textit{Case (i): $h(C_n) = 1$.} Then $\det(\coboundary) = 0$, so $\coboundary$ is not invertible. The kernel is $\ker \coboundary = \{(x,x,\ldots,x) : x \in \Ring\} \cong \Ring$. Since $\rank(\coboundary) = n-1$ (as the $(n-1) \times (n-1)$ upper-left submatrix is invertible), the Smith form is $\mathrm{diag}(1,\ldots,1,0)$ with $(n-1)$ ones. Thus $H^1 = \coker(\coboundary) \cong \Ring$ is free.

\textit{Case (ii): $h(C_n) - 1 \in \Ring^\times$.} Then $\det(\coboundary)$ is a unit, so $\coboundary$ is invertible. Thus $\ker \coboundary = 0$ and $\coker(\coboundary) = 0$.

\textit{Case (iii): $\valpi(h(C_n)-1) = a > 0$.} Write $h(C_n) - 1 = u\unif^a$ with $u \in \Ring^\times$. Then $\det(\coboundary) = u\unif^a$ is neither zero nor a unit. The equation $(h(C_n) - 1)x_0 = 0$ has only the trivial solution $x_0 = 0$ in the free module $\Ring$, so $\ker \coboundary = 0$. The Smith normal form is $\mathrm{diag}(1,\ldots,1,\unif^a)$ (with $(n-1)$ ones), giving $H^1 \cong \Ring/\unif^a$, a single bar of length $a$.
\end{proof}

\begin{corollary}
A cycle with holonomy $h(C)$ is consistent through $k$ digits of precision -- meaning all cocycles lift from $\sheaf{F}/\unif^k\sheaf{F}$ to $\sheaf{F}/\unif^{k+1}\sheaf{F}$ -- if and only if $h(C) \equiv 1 \pmod{\unif^{k+1}}$. The bar length $\valpi(h(C)-1)$ is the maximal precision at which consistency holds.
\end{corollary}

\begin{example}
\label{ex:triangle-holonomy}
Consider the triangle $C_3$ from Example~\ref{ex:triangle} with $\Ring = \Zp$, $\unif = p$, and edge units $m_{e_{12}} = 1-p$, $m_{e_{23}} = m_{e_{31}} = 1$. The holonomy is $h(C_3) = (1-p) \cdot 1 \cdot 1 = 1-p$, giving $\valpi(h(C_3)-1) = \valpi(-p) = 1$. By Theorem~\ref{thm:cycle-barcode}, $H^1(C_3;\sheaf{F}) \cong \Zp/p\Zp$ with barcode $\{[0,1)\}$: a single bar of length 1. The cycle is consistent modulo $p$ but fails at precision $p^2$. This matches the Smith normal form calculation showing $\coboundary$ has diagonal form $\mathrm{diag}(1,1,p)$.
\end{example}

For general graphs, cohomology is determined by holonomies around a cycle basis. If $G$ has first Betti number $\beta_1(G) = |E| - |V| + c$ where $c$ is the number of connected components, choose a maximal spanning forest and let $\{C_1, \ldots, C_{\beta_1}\}$ be the fundamental cycles obtained by adding back the remaining edges one at a time. When these cycles are edge-disjoint -- which holds generically, for instance when $G$ is a planar graph and the cycles are faces -- their holonomies contribute independently to torsion.

\begin{proposition}[Independent cycles under block decomposition]
\label{prop:independent-cycles}
Suppose there exists a spanning forest $T$ such that, after a gauge along $T$ and appropriate reordering of vertices and edges, the coboundary matrix $\coboundary$ is block diagonal with one block for each edge-disjoint fundamental cycle $C_i$. For a rank-1 unit sheaf,
\[
H^1_{\tors}(G;\sheaf{F}) \cong \bigoplus_{i=1}^\beta \Ring/(h(C_i)-1),
\]
with barcode $\{[0, \valpi(h(C_i)-1))\}_{i=1}^\beta$.
\end{proposition}

\begin{proof}
The block diagonal structure ensures that the Smith normal form of $\coboundary$ is the direct sum of the Smith normal forms of the individual blocks. Each block corresponds to a single cycle $C_i$, and by Theorem \ref{thm:cycle-barcode}, contributes $\Ring/(h(C_i)-1)$ to torsion. The result follows.
\end{proof}

%------------------------------------------
\subsection{Stability of Arithmetic Barcodes}
\label{ssec:barcode-stability}
%------------------------------------------

A fundamental question for any invariant derived from measured data is robustness under perturbation. For arithmetic barcodes, the relevant notion of proximity is $\unif$-adic distance on coboundary matrices: two coboundaries $\coboundary, \coboundary': C^0 \to C^1$ are close if they agree modulo high powers of $\unif$. The digit-SNF dictionary of Theorem~A immediately yields stability results, since digit maps at each level depend only on the coboundary reduced modulo $\unif^{k+1}$. We establish that congruence modulo $\unif^m$ completely determines all bar lengths shorter than $m$, with the threshold case -- where $m$ exceeds all bar lengths -- giving exact barcode preservation.

Throughout this subsection, write $\{a_j\}_{j=1}^r$ for the Smith normal form exponents of $\coboundary$ (the positive ones; zero exponents contribute no torsion), and recall from Theorem~A that
\[
d_k := \dim_{\mathbb{F}} \im(\Digit{k}) = \#\{j : 1 \le a_j \le k\}
\]
for $k \ge 0$, where $\mathbb{F} = \Ring/\unif$ is the residue field. The number of bars of length exactly $\ell$ is $d_\ell - d_{\ell-1}$ for $\ell \ge 1$, with $d_0 = 0$.

\begin{lemma}[Locality of the Digit Map]
\label{lem:digit-locality}
Fix $k \geq 0$. The digit connecting homomorphism
\[
\Digit{k}: H^0(G;\sheaf{F}/\unif^k\sheaf{F}) \longrightarrow H^1(G;\sheaf{F}/\unif\sheaf{F})
\]
depends only on the reduction of the coboundary operator modulo $\unif^{k+1}$. That is, if $\coboundary \equiv \coboundary' \pmod{\unif^{k+1}}$, then $\Digit{k}(\coboundary) = \Digit{k}(\coboundary')$.
\end{lemma}

\begin{proof}
The digit map arises as the connecting homomorphism in the long exact cohomology sequence induced by
\[
0 \to \unif^k\sheaf{F}/\unif^{k+1}\sheaf{F} \longrightarrow \sheaf{F}/\unif^{k+1}\sheaf{F} \longrightarrow \sheaf{F}/\unif^k\sheaf{F} \to 0.
\]
All three terms and their cochain complexes are obtained from $(C^\bullet(G;\sheaf{F}),\coboundary)$ by reduction modulo $\unif^{k+1}$. Replacing $\coboundary$ by any $\coboundary' \equiv \coboundary \pmod{\unif^{k+1}}$ produces the identical short exact sequence of complexes over $\Ring/\unif^{k+1}$. By functoriality of connecting homomorphisms with respect to maps of short exact sequences, the resulting digit maps coincide.
\end{proof}

We now prove Theorem~C from Section~\ref{sec:intro}.

\begin{proof}[Proof of Theorem C]
Assume $\coboundary \equiv \coboundary' \pmod{\unif^m}$ and fix $k < m$. Then $k+1 \le m$, so $\coboundary \equiv \coboundary' \pmod{\unif^{k+1}}$. By Lemma~\ref{lem:digit-locality}, the digit connecting maps coincide:
\[
\Digit{k}(\coboundary) = \Digit{k}(\coboundary'): H^0(G;\sheaf{F}/\unif^k\sheaf{F}) \longrightarrow H^1(G;\sheaf{F}/\unif\sheaf{F}).
\]
Hence $d_k(\coboundary) = \dim_{\mathbb{F}} \im(\Digit{k}(\coboundary)) = \dim_{\mathbb{F}} \im(\Digit{k}(\coboundary')) = d_k(\coboundary')$ for all $k < m$.

By the Digit-SNF Dictionary (Theorem~A),
\[
\#\{j : a_j(\coboundary) = \ell\} = d_\ell(\coboundary) - d_{\ell-1}(\coboundary) = d_\ell(\coboundary') - d_{\ell-1}(\coboundary') = \#\{j : a_j(\coboundary') = \ell\}
\]
for every $1 \le \ell < m$. Thus the truncated barcodes $\mathrm{Bar}^1_\unif(\coboundary) \cap [0,m) = \mathrm{Bar}^1_\unif(\coboundary') \cap [0,m)$ agree as multisets of intervals, and the truncated valuation persistence modules are isomorphic.
\end{proof}

\begin{corollary}[Nonexpansiveness]
\label{cor:nonexpansive}
If $\coboundary \equiv \coboundary' \pmod{\unif^m}$, then $d_k(\coboundary) = d_k(\coboundary')$ for every $k < m$. Equivalently, the truncated barcode map $\coboundary \mapsto \mathrm{Bar}^1_\unif(\coboundary) \cap [0,m)$ is locally constant on the $\unif^m$-adic neighborhood of $\coboundary$.
\end{corollary}

The most commonly applicable case occurs when the perturbation precision exceeds all torsion scales, yielding exact barcode preservation.

\begin{corollary}[Threshold Stability]
\label{cor:threshold-stability}
If $\coboundary \equiv \coboundary' \pmod{\unif^m}$ with $m > \max_j a_j(\coboundary)$, then the entire arithmetic barcode is preserved:
\[
\mathrm{Bar}^1_\unif(\coboundary) = \mathrm{Bar}^1_\unif(\coboundary'),
\]
hence the Smith normal form exponent multisets coincide and $H^1(G;\sheaf{F})_{\tors}(\coboundary) \cong H^1(G;\sheaf{F})_{\tors}(\coboundary')$ as $\Ring$-modules.
\end{corollary}

\begin{proof}
Apply Theorem~C with the observation that if every bar length $\ell$ satisfies $\ell < m$, then the truncated barcode $\mathrm{Bar}^1_\unif \cap [0,m)$ equals the full barcode. Since all multiplicities $\#\{j : a_j = \ell\}$ are preserved for $\ell < m$ and there are no bars of length $\ge m$, the complete multiset of exponents is determined. The isomorphism of torsion submodules follows from the structure theorem for finitely generated modules over a principal ideal domain.
\end{proof}

\begin{remark}[Optimality]
\label{rem:stability-optimal}
Theorem~C is best possible without additional assumptions. If $m \le \max_j a_j(\coboundary)$, then bars of length $\ge m$ can change arbitrarily under $\unif^m$-level perturbations. For a minimal example, consider the rank-1 case over $\Ring = \Zp$: the coboundaries $\coboundary = [p^m]$ and $\coboundary' = [0]$ differ by $p^m$ and have barcodes $\{[0,m)\}$ and $\emptyset$ respectively. Theorem~C correctly guarantees agreement of bar lengths shorter than $m$ (of which there are none for $\coboundary'$, and one of length exactly $m$ for $\coboundary$), but makes no claim about features at the threshold.
\end{remark}

\begin{proposition}[Determinantal Truncation Stability]
\label{prop:determinantal-truncation}
For $r \ge 1$, let 
\[
s_r(\coboundary) := \min\{\valpi(\det M) : M \text{ is an } r \times r \text{ minor of } \coboundary\}
\]
so that classically $s_r = a_1 + \cdots + a_r$ is the sum of the first $r$ Smith exponents (the $r$-th determinantal ideal). If $\coboundary \equiv \coboundary' \pmod{\unif^m}$, then for every $r \ge 1$,
\[
\min\{s_r(\coboundary), m\} = \min\{s_r(\coboundary'), m\}.
\]
In particular, if $s_r(\coboundary) < m$ then $s_r(\coboundary') = s_r(\coboundary)$, and conversely.
\end{proposition}

\begin{proof}
Fix $r$ and choose an $r \times r$ minor $M$ of $\coboundary$ with $\valpi(M) = s_r(\coboundary)$. The corresponding minor $M'$ of $\coboundary'$ satisfies $M' \equiv M \pmod{\unif^m}$, hence $\valpi(M') \ge \min\{\valpi(M), \valpi(M'-M)\}$ by the non-Archimedean property. If $\valpi(M) < m$ then $\valpi(M'-M) \ge m$, so the valuations differ and the ultrametric inequality is sharp: $\valpi(M') = \valpi(M)$. Thus $s_r(\coboundary') \le \valpi(M') = s_r(\coboundary)$. Symmetry in $\coboundary, \coboundary'$ gives equality. If both $s_r(\coboundary)$ and $s_r(\coboundary')$ are $\ge m$, the displayed equality of minima is tautological.
\end{proof}

\begin{remark}
Proposition~\ref{prop:determinantal-truncation} controls partial sums of exponents below the threshold via determinantal ideals, providing a classical commutative-algebraic perspective on truncation stability. However, Theorem~C is strictly sharper for recovering individual bar lengths, since the digit-SNF dictionary extracts the complete exponent multiset from digit ranks rather than just partial sums. In parametric families $\coboundary(\theta)$ where for each $r$ a single minor achieves valuation $s_r$ while all others have strictly larger valuation, the proposition yields local constancy of $s_r$ and hence of the barcode under generic non-degeneracy conditions.
\end{remark}

\begin{remark}[Valuated matroids and tropical geometry]
\label{rem:tropical-viewpoint}
The determinantal valuations appearing in Proposition~\ref{prop:determinantal-truncation} 
have a natural interpretation via valuated matroids \cite{DressWenzel1992_ValuatedMatroids}. 
For a matrix $\coboundary$ with entries in a discrete valuation ring, the valuations 
$w_B := \valpi(\det \coboundary_B)$ of all $r \times r$ minors $\coboundary_B$ define 
a valuated matroid of rank $r$ on the column set. The minima $s_r = \min_B w_B$ 
appearing in our stability result are precisely the minimum basis weights in this 
valuated matroid. Under $\unif^m$-level perturbations, if the perturbation valuation 
exceeds $m$ while at least one basis weight remains below $m$, the minimizing basis 
(and hence $s_r$) is unchanged -- this is the tropical linear space picture of 
\cite{SpeyerSturmfels2004_TropicalGrassmannian,Speyer2008_TropicalLinearSpaces}. 
The collection of all minor valuations defines a point in the tropical Grassmannian, 
and stability under $\unif^m$-perturbations corresponds to remaining in the same 
cell of the associated tropical linear space \cite{MaclaganSturmfels2015_TropicalGeometry}. 
This geometric perspective explains why our stability is governed by determinantal 
ideals and provides a combinatorial framework for understanding barcode persistence 
under perturbation.
\end{remark}
%------------------------------------------
\subsection{Computational Implications}
\label{ssec:computational-implications}
%------------------------------------------

The stability results provide both theoretical guarantees and practical algorithmic guidance for computing arithmetic barcodes from finite-precision data.

\textbf{Adaptive precision algorithms.}
Threshold stability enables iterative refinement strategies. Given edge data or coboundary entries measured to precision $m$ bits (in the $p$-adic case, working modulo $p^m$), compute the digit ranks $d_k$ for $k = 1, \ldots, m-1$ via linear algebra over the residue field $\mathbb{F} = \Ring/\unif$. This yields an initial barcode estimate with all bars of length $< m$ guaranteed correct by Theorem~C. If the longest detected bar has length $\ell_{\max} < m - \delta$ for some safety margin $\delta$, the computation is complete. If bars approach or reach the precision limit, additional measurements at higher precision are required, but only for degrees of freedom that participate in long-bar obstructions.

By Theorem~B, the saturation projector $\Pi_{\sat}$ identifies which components of the edge-cochain space contribute to torsion. In heterogeneous settings where measurement costs vary across edges, this suggests prioritized refinement: allocate higher precision to edges with large coefficients in $\im(\Pi_{\sat})$, while edges lying primarily in the free complement $\im(\Pi_{\free}) \cong H^1_{\free}$ tolerate coarser quantization. Row coefficients (evaluated via $\valpi$ of entries) of $\Pi_{\sat}$ provide per-edge leverage scores indicating sensitivity to precision.

\textbf{Dual computation and error checking.}
The digit-SNF dictionary provides two independent routes to the same invariant. The \emph{algebraic route} computes Smith normal form of $\coboundary$ directly over $\Ring$ via elementary row and column operations or Hensel lifting \cite{Kannan1979}, reading off exponents $\{a_j\}$ from the diagonal. Efficient algorithms for computing Smith normal forms of integer matrices are well-developed: Kannan and Bachem's polynomial-time algorithm \cite{Kannan1979} and near-optimal methods of Storjohann \cite{Storjohann1996_NearOptimalSNF} (with sparse refinements) pair naturally with $\unif$-adic Hensel refinement in our setting.

For sparse graphs or structured matrices, one approach may be significantly faster. For numerical robustness, computing both and verifying agreement provides error detection: if $d_k \neq \#\{j : 1 \le a_j \le k\}$ for some $k$, either the Smith form was computed incorrectly or digit map dimensions were miscounted. This redundancy is analogous to checksum validation in numerical linear algebra.

\textbf{Precision requirements for applications.}
In distributed systems where agents communicate quantized data over networks (Section~\ref{sec:application}), Corollary~\ref{cor:threshold-stability} provides a topology-dependent lower bound on communication precision. If the network graph $G$ with sheaf data $\sheaf{F}$ has arithmetic barcode with longest bar $\ell_{\max}$, then $b \ge \ell_{\max}$ bits per message suffice to detect all cycle inconsistencies, and $b < \ell_{\max}$ bits are insufficient. The barcode thus translates topological obstruction (cycles in the graph with nontrivial holonomy) into information-theoretic cost (bits required for global consistency).

For time-varying networks or sheaves evolving under dynamics, the barcode $\mathrm{Bar}^1_\unif(t)$ becomes time-dependent. Whenever $\coboundary(t) \equiv \coboundary(t') \pmod{\unif^m}$, the truncated barcode $\mathrm{Bar}^1_\unif \cap [0,m)$ is identical by Corollary~\ref{cor:nonexpansive}; it is locally constant on $\unif^m$-adic neighborhoods, with changes occurring only when the perturbation exits the $\unif^m$ ball corresponding to the next bar-length threshold. This enables tracking of persistent obstructions through parameter space.

\textbf{Stability in multiparameter settings.}
When combining geometric filtrations (varying graph topology or sheaf data by scale parameter $\tau$) with the algebraic precision filtration (varying $k$), one obtains two-parameter persistence indexed by $(\tau, k) \in \mathbb{R} \times \mathbb{N}$. Theorem~C provides vertical stability: fixing $\tau$, the barcode in the $k$-direction is robust to $\unif^m$-perturbations of the sheaf. Horizontal stability (varying $\tau$ at fixed $k$) follows from classical persistence stability for the geometric parameter. The interaction between axes -- how geometric features at different scales $\tau$ carry different precision signatures $k$ -- remains an open question for future investigation.

%------------------------------------------
\subsection{Outlook: Multiparameter Persistence and Higher-Rank Sheaves}
\label{ssec:barcode-outlook}
%------------------------------------------

The valuation filtration provides one axis of persistence; many applications involve simultaneous geometric filtrations. For a family $\{G_\tau\}_{\tau \in \mathbb{R}}$ of graphs (or a fixed graph with sheaf data varying by scale $\tau$), we obtain a two-parameter family $H^i(G_\tau; \sheaf{F})/\unif^k$ indexed by $(\tau, k) \in \mathbb{R} \times \mathbb{N}$. The associated rank functions
\[
S(\tau, k) := \dim_{\mathbb{F}} \gr^k H^1(G_\tau;\sheaf{F}), \qquad R(\tau, k) := \dim_{\mathbb{F}} \im\bigl(\Digit{k}(\tau)\bigr)
\]
define surfaces that are monotone in both parameters: nondecreasing in $\tau$ (features appear as scale increases) and nondecreasing in $k$ for the digit ranks $R(\tau,k)$ (more digits switch on as $k$ increases). These rank invariants summarize two-parameter persistence without requiring the full machinery of multiparameter persistence modules, which remains an active area of research.

For sheaves of rank greater than one, edge restrictions are matrix-valued. Holonomy around a cycle $C$ becomes a matrix product $H(C) = \prod_{e \in C} M_e \in \mathrm{GL}_n(\Ring)$. The obstruction to global compatibility is measured by $H(C) - I$, and the torsion structure arises from the Smith normal form of this difference. This \emph{matrix holonomy} perspective generalizes cycle barcodes to higher-rank sheaves and connects to representation theory of fundamental groups.

Applications naturally arise wherever ultrametric precision structure governs data. In sensor network synchronization, nodes measure relative phases or positions with limited precision; the arithmetic barcode measures consistency across the network and identifies cycles where precision bottlenecks occur. In financial networks tracking exchange rates, inconsistencies around currency cycles (triangular arbitrage) can be quantified via holonomy, with bar length measuring the precision to which no-arbitrage conditions hold. In geodetic networks or photometric calibration systems, measurements between stations accumulate errors; the barcode identifies which loops exhibit systematic drift and at what precision level. The framework extends to any setting where local measurements aggregate around cycles and precision limits detectability of global obstructions.

The foundations established here -- digit sequences, the Digit-SNF Dictionary, saturation splitting, and the barcode classification -- provide the algebraic infrastructure for arithmetic persistence theory. Full development of multiparameter invariants, general stability theorems (e.g., bottleneck distance bounds), algorithmic implementations, and detailed case studies will appear in subsequent work. The key insight is that torsion, often viewed as a computational nuisance in integral topology, becomes the central signal when data naturally stratifies by precision.

%--------------------------------------------------------------------
\section{Application: Distributed Consensus with Quantized Communication}
\label{sec:application}
%--------------------------------------------------------------------

Arithmetic barcodes have natural applications in distributed systems where agents communicate over networks using finite-precision messages. We outline how the p-adic cohomology framework provides information-theoretic bounds on communication requirements for achieving global consensus when local measurements are quantized.

%------------------------------------------
\subsection{Scalar Consensus: Clock Synchronization}
\label{ssec:scalar-consensus}
%------------------------------------------

Throughout this section, we work over the ring $\Ring = \mathbb{Z}_2$ of 2-adic integers with uniformizer $\unif = 2$, so that $\valpi(x) = \mathrm{val}_2(x)$ denotes the 2-adic valuation. Consider a network of $n$ autonomous agents positioned at vertices of a connected graph $G = (V,E)$, where each agent operates with a local clock running at rate $r_i \in \mathbb{R}_{>0}$. Agents can communicate along edges, measuring relative clock rates through direct comparison. For edge $e = \{i,j\}$, the measured rate ratio is $w_e := r_i/r_j \in \mathbb{Q}_{>0}$. In an ideal synchronized system, propagating these ratios around any cycle $C$ would yield the identity: $\prod_{e \in C} w_e = 1$, reflecting global consistency. In practice, measurement errors, calibration drift, or channel asymmetries cause the \emph{holonomy} $h(C) := \prod_{e \in C} w_e$ to deviate from unity.

The fundamental constraint is that all inter-agent communication uses finite-precision fixed-point arithmetic. Each rate ratio $w_e$ must be represented using $b$ bits, and propagated values accumulate quantization error. The question is: how many bits $b$ suffice to detect cycle inconsistencies?

We model this using a rank-1 unit sheaf over the ring $\mathbb{Z}_2$ of 2-adic integers. The choice $p=2$ is natural because $b$-bit fixed-point arithmetic corresponds to working modulo $2^b$, and the 2-adic valuation $\valpi(x)$ measures precision: if $\valpi(x) = k$, then $x$ is divisible by $2^k$ but not by $2^{k+1}$.

%\textbf{Normalizing to 2-adic units.}
Any positive rational $w_e \in \mathbb{Q}_{>0}$ can be written uniquely as $w_e = 2^{\kappa_e} u_e$ where $\kappa_e \in \mathbb{Z}$ and $u_e$ is a ratio of odd integers, hence a unit in $\mathbb{Z}_2$. To construct the network sheaf, we work with the unit parts $\{u_e\}$ and absorb the powers of 2 into vertex scalings. Specifically, fix a spanning tree $T \subseteq G$ and choose vertex scalings $\gamma_v = 2^{s_v}$ so that the gauge-transformed weights
\[
w'_e := \gamma_{h(e)}^{-1} w_e \gamma_{t(e)} = 2^{s_{t(e)} - s_{h(e)} + \kappa_e} u_e
\]
satisfy $w'_e \in \mathbb{Z}_2^\times$ for all $e \in T$. This is always possible by choosing the exponents $\{s_v\}$ to solve $s_{t(e)} - s_{h(e)} = -\kappa_e$ for $e \in T$, which has a unique solution up to adding a global constant. For edges not in $T$, if $\sum_{e \in C} \kappa_e = 0$ for all fundamental cycles $C$, then all transformed weights are units; otherwise some cycles contribute powers of 2 to their holonomy.

\begin{definition}
\label{def:clock-sheaf}
The \emph{clock synchronization sheaf} $\sheaf{F}$ on graph $G$ with gauge-normalized unit weights $\{u'_e \in \mathbb{Z}_2^\times\}_{e \in E}$ assigns:
\begin{itemize}
\item Vertex stalks: $\stalk{\sheaf{F}}{v} = \mathbb{Z}_2$ for all $v \in V$
\item Edge stalks: $\stalk{\sheaf{F}}{e} = \mathbb{Z}_2$ for all $e \in E$  
\item Restriction maps: For oriented edge $e: u \to v$ with unit weight $u'_e$,
\[
\res{\sheaf{F}}{u}{e}(x) = x, \qquad \res{\sheaf{F}}{v}{e}(x) = u'_e \cdot x
\]
\end{itemize}
\end{definition}

For a cycle $C$, the holonomy $h(C) = \prod_{e \in C} u'_e \in \mathbb{Z}_2^\times$ is now a 2-adic unit. We say the cycle is \emph{consistent through $k$ bits} if $h(C) \equiv 1 \pmod{2^k}$, meaning the holonomy agrees with the identity to $k$ bits of precision. The 2-adic valuation $a := \valpi(h(C) - 1)$ measures the maximum precision: $h(C)$ is consistent through $a$ bits but not through $a+1$ bits.

A cocycle $s \in H^0(G;\sheaf{F})$ represents a globally consistent assignment satisfying $u'_e s(u) = s(v)$ for all edges $e: u \to v$. By Theorem~\ref{thm:cycle-barcode}, if $G = C_n$ is a cycle graph with holonomy $h(C_n)$ and $\valpi(h(C_n)-1) = a$, then $H^1(C_n;\sheaf{F}) \cong \mathbb{Z}_2/2^a\mathbb{Z}_2$ with arithmetic barcode consisting of a single bar of length $a$. For a general graph with cycle basis $\{C_1, \ldots, C_\beta\}$, the torsion is controlled by the unit holonomies $\{h(C_i)\}$: their valuations bound the bar lengths; moreover, under a block decomposition (e.g., an edge-disjoint fundamental cycle set) the torsion splits as a direct sum with cyclewise exponents $\valpi(h(C_i)-1)$.

\begin{proposition}[Communication Precision for Cycle Detection]
\label{prop:scalar-precision}
Let $G$ be a connected graph with cycle basis $\{C_1, \ldots, C_\beta\}$, and let $\sheaf{F}$ be the clock synchronization sheaf with unit holonomies $h(C_i) \in \mathbb{Z}_2^\times$. Set $a_i := \valpi(h(C_i) - 1)$ for each cycle. Suppose each edge communicates a single $b$-bit integer representing its unit weight $u'_e$ modulo $2^b$, and all cycle products are computed modulo $2^b$.

Then $b \geq \max_i a_i$ is necessary and sufficient for detecting all cycle inconsistencies: every cycle $C_i$ with $h(C_i) \not\equiv 1 \pmod{2^b}$ is detectable, and cycles with $h(C_i) \equiv 1 \pmod{2^b}$ cannot be distinguished from perfectly consistent cycles at this precision level.

Moreover, if $b \geq \max_i a_i$, the complete arithmetic barcode can be recovered by computing $\dim_{\mathbb{F}_2} \im(\Digit{k})$ for $k = 1, \ldots, b$ using the digit maps of Theorem~A.
\end{proposition}

\begin{proof}
After the tree gauge normalization, each fundamental cycle $C_i$ contributes to cohomology through its holonomy $h(C_i) \in \mathbb{Z}_2^\times$. By the rank-1 cycle theorem (Theorem~\ref{thm:cycle-barcode}), the torsion from cycle $C_i$ is $\mathbb{Z}_2/2^{a_i}\mathbb{Z}_2$ where $a_i = \valpi(h(C_i)-1)$. The Digit-SNF Dictionary (Theorem~A) shows that to detect a bar of length $a_i$ requires computing cohomology modulo $2^{a_i}$. Working modulo $2^b$ reveals all bars of length at most $b$. Therefore $b \geq \max_i a_i$ is both necessary (to detect the longest bar) and sufficient (to detect all bars).

For necessity, if $b < a_i$ for some cycle $C_i$, then $h(C_i) \equiv 1 \pmod{2^b}$ since $h(C_i) - 1 \in 2^{a_i}\mathbb{Z}_2 \subseteq 2^b\mathbb{Z}_2$, making the inconsistency invisible at $b$-bit precision. For sufficiency, if $b \geq a_i$, then $h(C_i) \not\equiv 1 \pmod{2^b}$ (since $h(C_i) - 1 = 2^{a_i} u$ with $u \in \mathbb{Z}_2^\times$), allowing detection.
\end{proof}

\begin{remark}
The proposition addresses detection at a fixed communication precision $b$. Relating this to real-valued $\epsilon$-accuracy requires additional assumptions about dynamic range and the mapping from quantized integers to real rates, which depend on the specific consensus protocol. The cohomological bound provides a topology-dependent lower limit on communication precision that any protocol must respect.
\end{remark}

The saturation splitting of Theorem~B provides additional insight. The decomposition $C^1 = \sat(\im \coboundary) \oplus W$ separates edge measurements into those lying in $\sat(\im \coboundary)$, which propagate globally through the network and require high precision to resolve torsion obstructions, and those in the free complement $W \cong H^1_{\free}$, which represent degrees of freedom that do not participate in cycle inconsistencies. This suggests heterogeneous bit allocation: allocate more bits to components of the edge-cochain along $\im\Pi_{\sat}$ (e.g., via per-edge leverage scores derived from row norms of $\Pi_{\sat}$); components along the free complement $\im\Pi_{\free}\cong H^1_{\free}$ tolerate coarser quantization.

%------------------------------------------
\subsection{Vector-Valued Consensus and Matrix Holonomy}
\label{ssec:vector-consensus}
%------------------------------------------

The scalar framework extends to distributed computations over vector-valued data, where agents maintain state vectors $x_i \in \mathbb{R}^d$ and communicate linear combinations through quantized channels. This setting encompasses federated machine learning (gradient aggregation over parameter space), formation control in robotics (relative position estimation), distributed Kalman filtering (multi-sensor state fusion), and multi-agent coordination problems where local reference frames differ. Edge communications now involve $d \times d$ transformation matrices encoding coordinate changes, relative scalings, or linear compressions, and the resulting network sheaf has higher-rank stalks.

\begin{definition}
\label{def:vector-sheaf}
A \emph{vector consensus sheaf} $\sheaf{F}$ of rank $d$ over $\mathbb{Z}_2$ on graph $G$ assigns:
\begin{itemize}
\item Vertex stalks: $\stalk{\sheaf{F}}{v} = \mathbb{Z}_2^d$ for all $v \in V$
\item Edge stalks: $\stalk{\sheaf{F}}{e} = \mathbb{Z}_2^d$ for all $e \in E$
\item Restriction maps: For oriented edge $e: u \to v$, there are matrices $M_{u,e}, M_{v,e} \in \mathrm{GL}_d(\mathbb{Z}_2)$ giving
\[
\res{\sheaf{F}}{u}{e}(x) = M_{u,e} x, \qquad \res{\sheaf{F}}{v}{e}(x) = M_{v,e} x
\]
\end{itemize}
The \emph{relative transformation} along edge $e$ is $T_e := M_{v,e}^{-1} M_{u,e} \in \mathrm{GL}_d(\mathbb{Z}_2)$.
\end{definition}

For a cycle $C = (e_1, \ldots, e_k)$ with oriented edges, the \emph{matrix holonomy} is
\[
H(C) := T_{e_k} T_{e_{k-1}} \cdots T_{e_1} \in \mathrm{GL}_d(\mathbb{Z}_2).
\]
This measures the accumulated transformation when propagating a vector around the cycle. The obstruction to global consistency is the deviation $H(C) - I_d$ from the identity matrix. When $H(C) = I_d$ exactly, the cycle imposes no constraint; when $H(C) \approx I_d$ (close in the 2-adic metric), the cycle imposes weak constraints that become visible only at high precision.

Unlike the scalar case where $h(C) - 1 \in \mathbb{Z}_2$ is a single number with a single valuation, the matrix $H(C) - I_d \in M_d(\mathbb{Z}_2)$ encodes obstructions across multiple directions simultaneously. The Smith normal form over $\mathbb{Z}_2$ diagonalizes this matrix to reveal its torsion structure.

\begin{proposition}[Cycle contribution in rank $d$ under a block decomposition]
\label{prop:matrix-cycle}
Assume a tree gauge and an ordering of vertices/edges for which the coboundary matrix decomposes as a block direct sum over fundamental cycles (e.g., an edge-disjoint cycle basis). Then for each cycle $C$ the torsion contribution to $H^1(G;\sheaf{F})$ is $\mathrm{coker}\!\big(H(C)-I_d\big)$ as a $\mathbb{Z}_2$-module. Equivalently, if the Smith normal form of $H(C)-I_d$ has invariant factors $2^{a_1}, \ldots, 2^{a_r}$ (with $0\le a_1\le\cdots\le a_r$), then $C$ contributes bars of lengths $a_1, \ldots, a_r$.
\end{proposition}

\begin{proof}
In a tree gauge with a block decomposition, each fundamental cycle contributes a block whose 1–cochain relations are governed by $H(C)-I_d$. The torsion in the corresponding quotient is $\mathrm{coker}(H(C)-I_d)$, whose invariant factors give the bar lengths via the barcode classification.
\end{proof}

\begin{remark}
Without such a block decomposition, the global torsion is determined by the Smith normal form of the \emph{entire} coboundary; the matrices $H(C)-I_d$ still control bar lengths (they give sharp lower bounds and become exact after suitable basis choices), but contributions can couple across cycles.
\end{remark}

The Smith normal form exponents $\{a_j\}$ reveal anisotropic precision requirements: different directions in the $d$-dimensional space may require different numbers of bits to resolve inconsistencies. If all exponents are equal, $a_1 = \cdots = a_r$, the cycle is isotropically constrained. If they differ significantly, the cycle induces strong constraints in some principal directions (large $a_j$) and weak constraints in others (small $a_j$).

\begin{example}[Near-identity transformations; first-order regime]
\label{ex:near-identity}
Suppose a 3-cycle $C_3$ has edge transforms $T_{e_i} = I_d + 2^k A_i$ with $A_i \in M_d(\mathbb{Z}_2)$ and $k \ge 1$. Expanding the product $(I + 2^k A_1)(I + 2^k A_2)(I + 2^k A_3)$ modulo $2^{2k+1}$ gives
\[
H(C_3) - I_d \;=\; 2^k(A_1 + A_2 + A_3) \;+\; 2^{2k}(A_1 A_2 + A_2 A_3 + A_3 A_1) + O(2^{3k}).
\]
Set $A := A_1 + A_2 + A_3$ and factor: $H(C_3) - I_d = 2^k(A + 2^k B)$ where $B$ encodes the second-order terms. 

\textbf{Key observation:} If $A$ has Smith normal form $\mathrm{SNF}(A) = \mathrm{diag}(2^{\alpha_1}, \ldots, 2^{\alpha_s}, 1, \ldots, 1, 0, \ldots)$ with $0 \le \alpha_1 \le \cdots \le \alpha_s$ and $s$ nontrivial torsion factors, then for $k > \max_j \alpha_j$, the matrix $A + 2^k B$ has the same Smith normal form as $A$ over $\mathbb{Z}_2$. This follows because the $r \times r$ minors of $A + 2^k B$ satisfy $\det(M + 2^k N) \equiv \det(M) \pmod{2^{k+v}}$ where $v = \min\{\valpi(\det M'), M' \text{ any } r \times r \text{ minor of } A\}$, and when $k > \max_j \alpha_j$, the leading-order term dominates. Therefore
\[
\mathrm{SNF}(H(C_3) - I_d) = \mathrm{diag}(2^{k+\alpha_1}, \ldots, 2^{k+\alpha_s}, 2^k, \ldots, 2^k, 0, \ldots, 0),
\]
yielding $(d-s)$ bars of length $k$ and $s$ bars of lengths $k + \alpha_j$. This anisotropy reflects that different spatial directions have different sensitivity to the first-order perturbations $\{A_i\}$.
\end{example}

This matrix holonomy perspective applies naturally to several application domains:
\begin{itemize}
\item \textbf{Federated Learning.} Each agent $i$ computes a local gradient $g_i \in \mathbb{R}^d$ where $d$ is the model dimension. Agents communicate compressed or preconditioned gradients $M_{ij} g_i$ to neighbors, where $M_{ij} \in \mathbb{R}^{d \times d}$ encodes sketching operators, adaptive learning rates, or heterogeneous feature normalizations. After quantizing entries to $b$ bits and lifting to $\mathbb{Z}_2$-valued matrices, the holonomy around cycles of workers measures accumulated drift. Parameters with large Smith exponents in $H(C) - I$ are sensitive to quantization and require higher bit budgets, while those with small exponents tolerate coarser approximation.

\item \textbf{Formation Control.} Robots maintaining relative positions encode transformations $T_e \in \mathrm{SE}(d)$ (rigid motions) or $\mathrm{GL}_d(\mathbb{R})$ (affine transformations) between local reference frames. Discretizing to $\mathbb{Z}_2^d$ and computing holonomy around triangles or other cycles reveals orientation drift or scale inconsistencies. Bars of different lengths indicate that some spatial directions (e.g., along vs. perpendicular to the formation axis) have different precision requirements.

\item \textbf{Distributed Kalman Filtering.} Sensors fusing estimates of a state vector $x \in \mathbb{R}^d$ communicate information-weighted updates $\Sigma_i^{-1} x_i$ where $\Sigma_i$ is the local covariance matrix. The matrices $\Sigma_j^{-1} \Sigma_i$ along edges give relative weightings, and their 2-adic Smith exponents determine which state components need high-precision communication for optimal fusion.
\end{itemize}

In each case, the arithmetic barcode provides a topology-dependent lower bound on the bit complexity of achieving global consensus. Graphs with many independent cycles (large first Betti number $\beta_1$) accumulate more obstructions, while graphs with trivial cohomology (trees) have empty barcodes and achieve perfect consensus regardless of quantization.

%------------------------------------------
\subsection{Computational Approach}
\label{ssec:computation-consensus}
%------------------------------------------

For networks of modest size, the arithmetic barcode is computable in practice. Given the graph topology and measured transformation data (unit ratios $\{u'_e\}$ for scalars, matrices $\{T_e\}$ for vectors), one constructs the coboundary operator $\coboundary: C^0(G;\sheaf{F}) \to C^1(G;\sheaf{F})$ as an $n_1 \times n_0$ matrix over $\mathbb{Z}_2$ (or an $n_1 d \times n_0 d$ matrix for rank-$d$ sheaves), where $n_0 = |V|$ and $n_1 = |E|$. The Smith normal form over $\mathbb{Z}_2$ yields the invariant factors $\{2^{a_j}\}$ directly, giving the bar lengths.

For 2-adic computations, one works with rational or integer matrix entries and applies Hensel lifting: compute the Smith normal form modulo $2, 4, 8, \ldots$ iteratively until the exponents stabilize. Since the barcode is finite -- all bars have length bounded by the maximum torsion in $H^1(G;\sheaf{F})$ -- this stabilization occurs after finitely many doublings. For sparse graphs, standard sparse matrix algorithms accelerate the computation.

Alternatively, the digit map approach of Theorem~A provides a hierarchical algorithm: for $k = 1, 2, \ldots$, compute cohomology $H^0(G;\sheaf{F}/2^k\sheaf{F})$ over the finite ring $\mathbb{Z}_2/2^k\mathbb{Z}_2$ and record the dimension $d_k := \dim_{\mathbb{F}_2} \im(\Digit{k})$. The sequence $(d_k)_{k \geq 1}$ is nondecreasing and stabilizes once $k$ exceeds all bar lengths. The number of bars of length exactly $\ell$ is $d_\ell - d_{\ell-1}$, reconstructing the barcode from successive approximations using only linear algebra over finite fields.

Threshold stability (Corollary~\ref{cor:threshold-stability}) is crucial for robustness: if edge weights or transformations are measured to precision $m$ bits and $m$ exceeds the maximum bar length, the computed barcode is guaranteed correct. This enables iterative refinement -- initial low-precision measurements yield a coarse estimate of maximum bar length, which then guides where additional measurement precision is needed. For heterogeneous networks where measurement costs vary across edges, this provides a principled resource allocation strategy: edges participating in cycles with long bars merit higher precision, while edges in cohomologically trivial regions tolerate coarser quantization.

The framework extends to dynamic settings where network topology or edge data evolve over time, inducing a time-varying barcode $\mathrm{Bar}^1_2(t)$. When combined with geometric filtrations -- for instance, thresholding edges by signal-to-noise ratio or communication range -- this yields two-parameter persistence modules indexed by both topological scale and algebraic precision. Detailed stability theory, convergence analysis for specific consensus algorithms, and experimental validation on real distributed systems are directions for future work.
%--------------------------------------------------------------------
\section{Conclusion}
\label{sec:conclusion}
%--------------------------------------------------------------------

Our approach to the cohomology of network sheaves over discrete valuation rings focuses on the $p$-adic case $\Ring = \Zp$, but is not limited to it. The valuation filtration $\{p^k\}$ stratifies cochains and cohomology by algebraic precision, organizing torsion into a hierarchy measurable by digit connecting maps. Three main results form the backbone: the Digit-SNF Dictionary (Theorem~A) proves that dimensions of digit map images encode exactly the Smith normal form exponents, making the complete torsion structure both conceptually transparent and computationally accessible through linear algebra over the residue field. Saturation Splitting (Theorem~B) constructs explicit integral idempotents projecting onto canonical representatives for cohomology classes, valid over any DVR and requiring no geometric structure. Truncated Stability (Theorem~C) guarantees that arithmetic barcodes are robust under high-precision perturbations: when two coboundaries agree modulo $\unif^m$, their barcodes coincide on all bars of length less than $m$, ensuring the invariants are numerically stable for computations with measured data.

The digit-SNF correspondence suggests a natural notion of arithmetic persistence where torsion summands $\Ring/\unif^a$ become bars of length $a$, measuring precision thresholds at which cohomology classes fail to lift. For rank-one sheaves, cycle holonomy determines bar lengths explicitly: $\valpi(h(C)-1)$ quantifies how many digits of precision a loop remains consistent. Threshold stability guarantees barcode invariance when perturbations are smaller than all bar lengths, providing robustness for numerical computation. This reframes torsion from computational obstacle to primary signal in settings where data naturally stratifies by precision.

Several mathematical questions remain open. Rigorous stability theory for arithmetic barcodes -- including bottleneck distance bounds and interleaving formulations -- requires careful development of ultrametric persistence module theory. Algorithmic complexity analysis should clarify when the digit route (successive mod $\unif^k$ computations) outperforms direct Smith normal form, particularly for sparse or structured matrices. Extension to cellular sheaves of dimension greater than one involves additional subtleties from higher differentials in the Bockstein spectral sequence and potential interactions with cup products. Functoriality and naturality properties under graph morphisms and sheaf pullbacks deserve systematic exposition. For multiparameter persistence combining geometric and algebraic filtrations, the interaction between scale and precision axes -- how topological features at different scales carry different precision signatures -- merits careful investigation.

Our initial framework positions network sheaf cohomology to address problems involving hierarchical precision structure: error-correcting codes where codewords degrade through noise channels, differential privacy mechanisms that release data at controlled resolution, cryptographic protocols with information-theoretic security bounds, synchronization networks with measurement uncertainty, and sensor calibration systems accumulating drift. The unifying feature is ultrametric distance -- whether from $p$-adic topology, valuation hierarchies, or quantized approximation -- inducing natural filtrations whose persistent features measure global consistency at varying precision. We anticipate that algebraic tools from module theory over DVRs, combined with topological methods from persistent homology, will provide new invariants and algorithms for such systems. 

Our prototype distributed consensus application demonstrates that arithmetic persistence addresses concrete engineering problems where ultrametric precision hierarchies arise naturally. The examples sketched here -- clock synchronization, formation control, federated learning -- represent only initial directions. The purely algebraic nature of our main results hints at broader applicability: any setting where data lives over a discrete valuation ring and propagates through a network can leverage these tools.

% Bibliography
%\printbibliography

\end{document}